\documentclass{article}

\usepackage[T1]{fontenc}
\usepackage{times}
\usepackage{epsfig,graphicx,tabularx,color}
\usepackage[cp850]{inputenc}
\usepackage[english]{babel}
\usepackage{cite}
\usepackage{amsmath,amssymb,amsfonts,amsthm,mathrsfs,comment,mathdots,multirow,tikz}
\usepackage{rotating, longtable, relsize}
\usepackage[f]{esvect}
\usepackage{times}
\usepackage{fancyhdr}
\usepackage{lipsum}
\usepackage{nicefrac}

\usetikzlibrary{calc,3d}

\newtheorem{theorem}{Theorem}[section]

\newtheorem{definition}[theorem]{Definition}
\newtheorem{proposition}[theorem]{Proposition}

\theoremstyle{definition}

\def\XXint#1#2#3{{\setbox0=\hbox{$#1{#2#3}{\int}$}
     \vcenter{\hbox{$#2#3$}}\kern-.5\wd0}}

\def\ud{{\rm\,d}}

\def\C{\mathbb{C}}

\def\N{\mathbb{N}}

\def\R{\mathbb{R}}

\def\DD{\mathcal{D}}
\def\EE{\mathcal{E}}

\def\OO{\mathcal{O}}

\def\pr(#1){\left({#1}\right)}
\def\br[#1]{\left[{#1}\right]}

\def\abs#1{\left|{#1}\right|}
\def\norm#1{\left\|{#1}\right\|}

\def\pFq#1#2{{\,}_{#1}F_{#2}}

\def\ii{{\rm i}}

\newcommand{\bs}[1]{\boldsymbol{#1}}
\newcommand{\floor}[1]{\left\lfloor #1 \right\rfloor}

\usepackage{hyperref}
\usetikzlibrary{arrows,positioning}

\begin{document}

\title{Computing ${}_{q+1}{F}_q$ via rational conformal maps}
% Short title for running heads:

\author{%
{\sc John Lopez\thanks{Email: John.LopezSantander@umanitoba.ca}~~and Richard Mika\"el Slevinsky\thanks{Email: Richard.Slevinsky@umanitoba.ca}}\\[2pt]
Department of Mathematics, University of Manitoba, Winnipeg, Canada}
% Short list of authors for running heads:

\maketitle

%%%%%%%%%%%%%%%%%%%%%%%%%%%%%%%%%%%%%%%%%%%%%%%%%%%%%%%%%%%%%%%%%%%%%%%%%%%%%%%%%%%%%%%%%%%%%%%%%%%%%%%%%%%%%%%%
%%%%%%%%%%%%%%%%%%%%%%%%%%%%%%%%%%%%%%%%%%%%%%%%%%%%%%%%%%%%%%%%%%%%%%%%%%%%%%%%%%%%%%%%%%%%%%%%%%%%%%%%%%%%%%%%
%%%%%%%%%%%%%%%%%%%%%%%%%%%%%%%%%%%%%%%%%%%%%%%%%%%%%%%%%%%%%%%%%%%%%%%%%%%%%%%%%%%%%%%%%%%%%%%%%%%%%%%%%%%%%%%%
%%%%%%%%%%%%%%%%%%%%%%%%%%%%%%%%%%%%%%%%%%%%%%%%%%%%%%%%%%%%%%%%%%%%%%%%%%%%%%%%%%%%%%%%%%%%%%%%%%%%%%%%%%%%%%%%
\begin{abstract}
% Body of abstract:
It is well-known that certain M\"obius (linear fractional) transformations provide analytic continuation of the Gauss hypergeometric function. These transformations work by exchanging the roles of the three regular singular points and $0$, $1$, and $\infty$ of the Gauss hypergeometric differential equation. Not one of these transformations, however, brings the points $e^{\pm{\rm i}\pi/3}$ closer to the origin, and convergence in a neighbourhood of these is also troublesome. In this work, we produce a Maclaurin series in optimal complexity after rational conformal mapping that induces pointwise convergence to any ${}_{q+1}F_q$ throughout the cut complex plane $\mathbb{C}\setminus[1,\infty)$. We show that for the optimal parameter values this series converges faster than an approximate mapped Chebyshev series obtained by the ultraspherical spectral method.

% Keywords:
\end{abstract}
% MSC classifications: 30C30 Schwarz--Christoffel-type mapping, 33C20 pFq, 33F05 computational aspects of special functions, 41A20 rational approx, 65D20 computation of special functions, 65E10 numerical methods in conformal mappings

%%%%%%%%%%%%%%%%%%%%%%%%%%%%%%%%%%%%%%%%%%%%%%%%%%%%%%%%%%%%%%%%%%%%%%%%%%%%%%%%%%%%%%%%%%%%%%%%%%%%%%%%%%%%%%%%
%%%%%%%%%%%%%%%%%%%%%%%%%%%%%%%%%%%%%%%%%%%%%%%%%%%%%%%%%%%%%%%%%%%%%%%%%%%%%%%%%%%%%%%%%%%%%%%%%%%%%%%%%%%%%%%%
%%%%%%%%%%%%%%%%%%%%%%%%%%%%%%%%%%%%%%%%%%%%%%%%%%%%%%%%%%%%%%%%%%%%%%%%%%%%%%%%%%%%%%%%%%%%%%%%%%%%%%%%%%%%%%%%
%%%%%%%%%%%%%%%%%%%%%%%%%%%%%%%%%%%%%%%%%%%%%%%%%%%%%%%%%%%%%%%%%%%%%%%%%%%%%%%%%%%%%%%%%%%%%%%%%%%%%%%%%%%%%%%%
\section{Introduction}

Generalized hypergeometric functions \cite[\S 16]{Olver-et-al-NIST-10} are formally equal to their Maclaurin series:
\begin{equation}\label{eq:pFqbyMac}
\pFq{p}{q}\left(\begin{array}{c}
\alpha_1,\ldots,\alpha_p\\
\beta_1,\ldots,\beta_q\end{array};
z\right) = \sum_{n=0}^\infty \dfrac{(\alpha_1)_n\cdots(\alpha_p)_n}{(\beta_1)_n\cdots(\beta_q)_n}\dfrac{z^n}{n!},
\end{equation}
where $(\alpha)_n = \Gamma(\alpha+n)/\Gamma(x)$ is the Pochhammer symbol for the rising factorial~\cite[\S 6.1.22]{Abramowitz-Stegun-65}. If some $\alpha_j \in -\N_0$, then the series terminates after at most $|\alpha_j|+1$ terms. The Maclaurin series has radius of convergence $0$ if $p > q+1$, $1$ if $p = q+1$, and $\infty$ if $p\le q$.

If $p\ge q+1$, the Maclaurin series may be analytically continued into an appropriately slit complex plane by the Meijer $G$-function~\cite[\S 16.18.1]{Olver-et-al-NIST-10}, or by algorithms that solve the generalized hypergeometric differential equation:
\begin{equation}\label{eq:pFqODE}
\frac{\ud}{\ud z}\prod_{i=1}^q\left(z\frac{\ud}{\ud z} + \beta_i - 1\right) u(z) = \prod_{i=1}^p\left(z\frac{\ud}{\ud z} + \alpha_i\right) u(z),\qquad u(0) = 1.
\end{equation}
When $p = q+1$, the orders of the differential operators on both sides of the equation are the same, $q+1$. Through Eq.~\eqref{eq:pFqODE}, it can be seen that there are three regular singular points, at $0$, $1$, and $\infty$. Certain M\"obius (linear fractional) transformations analytically continue the Gauss hypergeometric function, $\pFq{2}{1}$, to the whole cut plane $\C\setminus[1,\infty)$. These transformations represent a given $\pFq{2}{1}$ as a linear combination of at most two other $\pFq{2}{1}$'s with arguments that can generally be brought closer to the origin. Exchanging the roles of $0$, $1$, and $\infty$, however, is rather restrictive to the M\"obius transformations, causing $\{e^{\pm{\rm i}\pi/3}\}$ to be mapped to itelf under any of them, rendering numerical computation challenging in neighbourhoods of these fixed points.

Many numerical algorithms have been developed to compute generalized hypergeometric functions~\cite{Forrey-137-79-97,Becken-Schmelcher-126-449-00,Muller-90-179-01,Michel-Stoitsov-178-535-08,Colman-Cuyt-van-Deun-38-11-1-11,Willis-59-447-12,Doornik-84-1813-15,Pearson-Olver-Porter-74-821-17,Crespo-et-al-84-1-20,Johansson-45-30-1-19,Slevinsky-98-587-25}, with the majority focusing special cases with $p+q\le3$. Moreover, several techniques in analytic continuation have been derived, including the re-expansion formul\ae~of B\"uhring~\cite{Buhring-18-884-87,Buhring-19-1249-88}, N{\o}rlund~\cite{Norlund-94-289-55}, and the Lopez group~\cite{Lopez-Temme-39-349-13,Lopez-Pagola-Sinusia-421-982-15,Lopez-Pagola-Palacios-55-891-22}.

In this work, we present a numerical strategy to compute $\pFq{q+1}{q}$ with a rational conformal map that allows a Maclaurin series in the transformed variable to provide exponential convergence throughout the cut complex plane $\C\setminus[1,\infty)$. We draw inspiration from two sources: the polynomial conformal map in the recent work of Jacobs, Piret, and Fornberg~\cite{Jacobs-Piret-Fornberg-1-26} and the rational conformal map in the earlier work of Skorokhodov~\cite{Skorokhodov-502-636-03}. Secondly, we will show how the factored structure of the generalized hypergeometric differential equation ensures that an optimal complexity of ${\cal O}(qmn)$ flops is attained, where $n$ is the truncation degree and $1\le m\le 4$ is an integer map parameter. And thirdly, we shall provide a generalization of the Euler transformation~\cite[\S 15.8.1]{Olver-et-al-NIST-10} to increase the numerical stability of positively balanced functions in the neighbourhood of argument unity. We include a collection of numerical examples to demonstrate the effectiveness of the new approach. Our method is implemented in the {\sc Julia} package {\tt HypergeometricFunctions.jl}~\cite{Slevinsky-GitHub-HypergeometricFunctionsjl}.

%%%%%%%%%%%%%%%%%%%%%%%%%%%%%%%%%%%%%%%%%%%%%%%%%%%%%%%%%%%%%%%%%%%%%%%%%%%%%%%%%%%%%%%%%%%%%%%%%%%%%%%%%%%%%%%%
%%%%%%%%%%%%%%%%%%%%%%%%%%%%%%%%%%%%%%%%%%%%%%%%%%%%%%%%%%%%%%%%%%%%%%%%%%%%%%%%%%%%%%%%%%%%%%%%%%%%%%%%%%%%%%%%
%%%%%%%%%%%%%%%%%%%%%%%%%%%%%%%%%%%%%%%%%%%%%%%%%%%%%%%%%%%%%%%%%%%%%%%%%%%%%%%%%%%%%%%%%%%%%%%%%%%%%%%%%%%%%%%%
%%%%%%%%%%%%%%%%%%%%%%%%%%%%%%%%%%%%%%%%%%%%%%%%%%%%%%%%%%%%%%%%%%%%%%%%%%%%%%%%%%%%%%%%%%%%%%%%%%%%%%%%%%%%%%%%

\subsection{The polynomial conformal map of Jacobs, Piret, and Fornberg}

Recently, Jacobs, Piret, and Fornberg~\cite{Jacobs-Piret-Fornberg-1-26} constructed Maclaurin series to the Gauss hypergeometric function with a polynomial conformal map as argument:
\[
\pFq{2}{1}\left(\begin{array}{c}\alpha_1~\alpha_2\\\beta_1\end{array}; z = 1-(1-w)^m\right) \quad\Leftrightarrow\quad w = a(z;m) = 1-\sqrt[m]{1-z},\qquad m\in\N.
\]
As they show, the polynomial conformal map transforms the Gauss hypergeometric differential equation to:
\begin{align}\label{eq:ATODE}
u(0) & = 1,\nonumber\\
(1-w)\left[1 - (1-w)^m\right] u''(w) + \left\{ \beta_1m - \left[1-(1-w)^m\right] \left[(\alpha_1+\alpha_2)m+1\right] \right\} u'(w) - \alpha_1\alpha_2m^2(1-w)^{m-1}u(w) & = 0.
\end{align}
This transformation has a number of favourable properties. Firstly, the polynomial structure of the map ensures that an $(m+1)$-term recurrence relation exists for the coefficients in the Maclaurin series solution to Eq.~\eqref{eq:ATODE}, making the $n$-term truncation cost $\OO(mn)$ flops. Secondly, the convergence set $\{z\in\C: \abs{1-\sqrt[m]{1-z}} < 1\}$ of the algebraic transformation is much larger than the unit disk, and is increasing in $m$. Nevertheless, this set is bounded due to the algebraic growth in the transformation, reaching its zenith on the negative real axis at the point $z = 1-2^m$. To determine the radius of convergence of the corresponding Maclaurin series, the roots of the leading polynomial, $(1-w)\left[1-(1-w)^m\right]$ are analyzed: it is determined that the singular points are all regular and, apart from the one at the origin, are outside the unit disk $\{w\in\C:\abs{w}< 1\}$ for $2\le m\le 6$, and they report that $m=4$ provides the best balance between accuracy and efficiency. The crucial observation in their work is that despite adding regular singular points to the Gauss differential equation, the Maclaurin series in the transformed variable converges far beyond the unit disk in the original coordinates.

%%%%%%%%%%%%%%%%%%%%%%%%%%%%%%%%%%%%%%%%%%%%%%%%%%%%%%%%%%%%%%%%%%%%%%%%%%%%%%%%%%%%%%%%%%%%%%%%%%%%%%%%%%%%%%%%
%%%%%%%%%%%%%%%%%%%%%%%%%%%%%%%%%%%%%%%%%%%%%%%%%%%%%%%%%%%%%%%%%%%%%%%%%%%%%%%%%%%%%%%%%%%%%%%%%%%%%%%%%%%%%%%%
%%%%%%%%%%%%%%%%%%%%%%%%%%%%%%%%%%%%%%%%%%%%%%%%%%%%%%%%%%%%%%%%%%%%%%%%%%%%%%%%%%%%%%%%%%%%%%%%%%%%%%%%%%%%%%%%
%%%%%%%%%%%%%%%%%%%%%%%%%%%%%%%%%%%%%%%%%%%%%%%%%%%%%%%%%%%%%%%%%%%%%%%%%%%%%%%%%%%%%%%%%%%%%%%%%%%%%%%%%%%%%%%%

\subsection{Analytic continuation by rational transformations}\label{subsection:analyticcontinuationbyrationaltransformations}

Rational transformations have a rich history in analytically continuing generalized hypergeometric functions beyond the unit disk. We begin with transformations with analytic coefficients. Firstly, we have the continuation formula at infinity~\cite[\S 16.8]{Olver-et-al-NIST-10}:
\begin{align*}
& \pFq{q+1}{q}\left(\begin{array}{c} \alpha_1,\ldots,\alpha_{q+1}\\ \beta_1,\ldots,\beta_q\end{array}; z\right)
=
\sum_{j=1}^{q+1}\frac{\prod_{k=1, k\ne j}^{q+1}\frac{\Gamma(\alpha_k-\alpha_j)}{\Gamma(\alpha_k)}}{\prod_{k=1}^q\frac{\Gamma(\beta_k-\alpha_j)}{\Gamma(\beta_k)}}\frac{1}{(-z)^{\alpha_j}} \times\\
&\qquad \pFq{q+1}{q}\left(\begin{array}{c} \alpha_j, 1-\beta_1+\alpha_j,\ldots 1-\beta_q+\alpha_j\\ 1-\alpha_1+\alpha_j,\ldots, *, \ldots, 1-\alpha_{q+1}+\alpha_j\end{array};\frac{1}{z}\right),
\end{align*}
where the $*$ indicates that the entry $1-\alpha_j+\alpha_j$ is omitted. B\"uhring~\cite{Buhring-18-884-87,Buhring-19-1249-88} builds on this classical result with an expansion in shifted reciprocal powers:
\begin{align*}
& \pFq{q+1}{q}\left(\begin{array}{c} \alpha_1,\ldots, \alpha_{q+1}\\ \beta_1,\ldots,\beta_q\end{array}; z\right) = \sum_{j=1}^{q+1}\frac{(z_0-z)^{-\alpha_j}}{\Gamma(\alpha_j)}\sum_{n=0}^\infty \frac{\Gamma(\alpha_j+n)}{n!}\frac{\prod_{k=1,k\ne j}^{q+1}\frac{\Gamma(\alpha_k-\alpha_j-n)}{\Gamma(\alpha_k)}}{\prod_{k=1}^q\frac{\Gamma(\beta_k-\alpha_j-n)}{\Gamma(\beta_k)}}\\
& \qquad\qquad \times \pFq{q+1}{q}\left(\begin{array}{c} \alpha_1-\alpha_j-n,\ldots,\alpha_{q+1}-\alpha_j-n\\ \beta_1-\alpha_j-n,\ldots, \beta_q-\alpha_j-n\end{array};z_0\right)(z-z_0)^{-n}.
\end{align*}
N{\o}rlund~\cite[Eq. (1.21)]{Norlund-94-289-55} produced a continuation formula that converges in the left half-plane $\Re z < \frac{1}{2}$. More recently, the L\'opez group~\cite{Lopez-Temme-39-349-13,Lopez-Pagola-Sinusia-421-982-15,Lopez-Pagola-Palacios-55-891-22} has introduced a collection of continuation formulas based on one- or two-point Taylor expansions in certain integral representations, the first of which generalizes the N\o rlund series to converge in any rotated left half-plane. For any $w\in\C$, the expansion:
\begin{align*}
\pFq{q+1}{q}\left(\begin{array}{c} \alpha_1,\ldots, \alpha_{q+1}\\ \beta_1,\ldots, \beta_q\end{array};z\right) = (1-wz)^{-\alpha_1}\sum_{k=0}^\infty \frac{(\alpha_1)_k}{k!}\pFq{q+1}{q}\left(\begin{array}{c} -k, \alpha_2,\ldots,\alpha_{q+1}\\ \beta_1,\ldots,\beta_q\end{array};\frac{1}{w}\right)\left(\frac{wz}{wz-1}\right)^k,
\end{align*}
converges in the region $\{z\in\C : \abs{1-wz} > \max\{\abs{w},\abs{1-w}\}\abs{z}\}$.

\subsection{The rational conformal maps of Skorokhodov}

In~\cite{Skorokhodov-502-636-03}, Skorokhodov discussed the numerical evaluation of $\pFq{q+1}{q}$ via Maclaurin series and linear- and quadratic-fractional transformations. Given a unimodular complex number $\gamma\ne1$, the linear-fractional transformation considers:
\[
\pFq{q+1}{q}\left(\begin{array}{c}
\alpha_1,\ldots,\alpha_{q+1}\\
\beta_1,\ldots,\beta_q\end{array};
z = 1-\frac{1-w}{1-\gamma w}\right) \quad\Leftrightarrow\quad w = \frac{z}{1-\gamma(1-z)}.
\]
With parameter $\gamma$, an optimal value is determined to minimize $\abs{w(z)}$, achieving the fastest exponential convergence in this class of transformations. This value is:
\[
\gamma_{\rm opt} = -\frac{\abs{1-z}}{1-z},
\]
and, in so choosing, it ensures convergence in the cut complex plane $\C\setminus[1,\infty)$, something rather remarkable since for every fixed $\gamma$, the region of convergence is only a half-plane! For example, for $\gamma=-1$, the corresponding half-plane is $\Re z < 1$.

Again with unimodular $\gamma\ne1$, the quadratic-fractional transformation considers:
\[
\pFq{q+1}{q}\left(\begin{array}{c}
\alpha_1,\ldots,\alpha_{q+1}\\
\beta_1,\ldots,\beta_q\end{array};
z = 1-\left(\frac{1-w}{1-\gamma w}\right)^2\right) \quad\Leftrightarrow\quad w = \frac{1-\sqrt{1-z}}{1-\gamma\sqrt{1-z}},
\]
though Skorokhodov's original formula~\cite[Eq. (4)]{Skorokhodov-502-636-03} does not make the square and M\"obius transformations appear so clearly, and thereby the connection with the linear-fractional transformation. Skorokhodov found that this map is the Schwarz--Christoffel map from the unit disk to the cut complex plane, constructed in general as follows.
\begin{theorem}[Theorem 2.2 in~\cite{Driscoll-Trefethen-02}]
Let $P$ be the interior of a polygon $\Gamma$ having vertices $z_1,\ldots,z_n$ and interior angles $\alpha_1\pi,\ldots,\alpha_n\pi$ in counterclockwise order. Let $f$ be any conformal map from the unit disk $\{w\in\C:\abs{w}<1\}$ to $P$. Then:
\[
f(w) = A + C\int^w\prod_{k=1}^n\left(1-\frac{\zeta}{w_k}\right)^{\alpha_k-1}\ud\zeta,
\]
for some complex constants $A$ and $C$, where $z_k = f(w_k)$ for $k=1,\ldots,n$.
\end{theorem}
Now, if we define the infinite $2$-gon with vertices $z_1 = 1$ and $z_2 = \infty$ and interior angles $\alpha_1 = 2$ and $\alpha_2 = -2$, respectively, and we fix the first prevertex on the unit circle at $w_1 = 1$, then:
\begin{equation}\label{eq:SchwarzChristoffelSingleCut}
f(w) = A + C\int^w \frac{1-\zeta}{(1-\zeta/w_2)^3}\ud\zeta = A + C\frac{w_2}{2}\frac{1+w_2-2w}{(1-w/w_2)^2}.
\end{equation}
Choosing the conformal center to be the origin, $f(0) = 0$, and fixing $f(1) = 1$ and letting $\gamma=w_2^{-1}$, we recover $f(w)$ as Skorokhodov's quadratic-fractional transformation. Thus it is clear that the quadratic-fractional transformation is a biholomorphic conformal map from the open unit disk onto the cut complex plane $\C\setminus[1,\infty)$. Moreover, there is an analogous optimal parameter value, equivalently an optimal prevertex, that minimizes $\abs{w(z)}$:
\[
\gamma_{\rm opt} = -\frac{\abs{\sqrt{1-z}}}{\sqrt{1-z}}.
\]

%%%%%%%%%%%%%%%%%%%%%%%%%%%%%%%%%%%%%%%%%%%%%%%%%%%%%%%%%%%%%%%%%%%%%%%%%%%%%%%%%%%%%%%%%%%%%%%%%%%%%%%%%%%%%%%%
%%%%%%%%%%%%%%%%%%%%%%%%%%%%%%%%%%%%%%%%%%%%%%%%%%%%%%%%%%%%%%%%%%%%%%%%%%%%%%%%%%%%%%%%%%%%%%%%%%%%%%%%%%%%%%%%
%%%%%%%%%%%%%%%%%%%%%%%%%%%%%%%%%%%%%%%%%%%%%%%%%%%%%%%%%%%%%%%%%%%%%%%%%%%%%%%%%%%%%%%%%%%%%%%%%%%%%%%%%%%%%%%%
%%%%%%%%%%%%%%%%%%%%%%%%%%%%%%%%%%%%%%%%%%%%%%%%%%%%%%%%%%%%%%%%%%%%%%%%%%%%%%%%%%%%%%%%%%%%%%%%%%%%%%%%%%%%%%%%

\section{Our rational conformal map}\label{section:rationalconformalmap}

Given a unimodular complex number $\gamma\ne1$ and a natural number $m\in\N$, consider the Maclaurin series to the generalized hypergeometric function with the rational conformal map as argument:
\begin{equation}\label{eq:RAT}
\pFq{q+1}{q}\left(\begin{array}{c}
\alpha_1,\ldots,\alpha_{q+1}\\
\beta_1,\ldots,\beta_q\end{array};
z = 1-\left(\frac{1-w}{1-\gamma w}\right)^m\right) \quad\Leftrightarrow\quad w = r(z; \gamma, m)= \frac{1-\!\!\sqrt[m]{1-z}}{1-\gamma\!\!\sqrt[m]{1-z}}.
\end{equation}
Observe that Eq.~\eqref{eq:RAT} draws from both the polynomial map of Jacobs, Piret, and Fornberg~\cite{Jacobs-Piret-Fornberg-1-26}, $r(z; 0, m) \equiv a(z; m)$, and the linear- and quadratic-fractional maps of Skorokhodov~\cite{Skorokhodov-502-636-03}. It has a number of useful properties, which we will describe below.

Firstly, it follows from the Schwarz--Christoffel biholomorphism that $\abs{r(z;\gamma,2)} \le 1$ for $z\in\C$ and the bound is only attained on the cut $[1,\infty)$. Secondly, there is an analogous optimal parameter value for every $m$ to the linear- and quadratic-fractional maps of Skorokhodov that maximizes the magnitude of the denominator:
\[
\gamma_{\rm opt} = -\frac{\abs{\!\!\sqrt[m]{1-z}}}{\!\!\sqrt[m]{1-z}},\quad\Rightarrow\quad r(z;\gamma_{\rm opt}, m) = \frac{1-\!\!\sqrt[m]{1-z}}{1+\abs{\!\!\sqrt[m]{1-z}}}.
\]
Thirdly, to avoid subtractive cancellation near the origin, $r(z;\gamma,m)$ can be expressed as:
\[
r(z;\gamma,m) = \frac{z}{(1-\gamma\!\!\sqrt[m]{1-z})[1+\!\!\sqrt[m]{1-z} + (\!\!\sqrt[m]{1-z})^2 + \cdots + (\!\!\sqrt[m]{1-z})^{m-1}]},
\]
by the difference of powers.

Figure~\ref{fig:contourf_conformalmaps} illustrates the modulus of $a(z;m)$, $r(z;-1,m)$ and $r(z;\gamma_{\rm opt},m)$ for $m=2,3,4$, where it is easy to see that the rational-algebraic transformations are uniformly bounded over the whole complex plane. It is also apparent that both parameters $\gamma$ and $m$ have a strong effect on the modulus, findings we have summarized as follows, with the proof relegated to Appendix~\ref{Appendix:Inequalities}.

\begin{theorem}\label{theorem:Inequalities}
Consider $r(z;\gamma,m)$ defined in Eq.~\eqref{eq:RAT} defined by the principal branch of the $m^{th}$ root.
Then, for $2\le m < n$:
\[
\abs{r(z;-1,n)} \le \abs{r(z;\gamma_{\rm opt},m)} \le \abs{r(z;-1,m)} \le 1,\qquad \forall z\in\C.
\]
\end{theorem}

While using $\gamma_{\rm opt}$ may appear to be attractive, there are two drawbacks. Firstly, by requiring the modulus, a Maclaurin series in this transformed variable would represent an analytic function in a series of non-analytic ones, potentially complicating further use of the approximations. Secondly, an analysis of the additional regular singular points that the rational map introduces into the transformed differential equation, similar to that of Jacobs, Piret, and Fornberg~\cite{Jacobs-Piret-Fornberg-1-26}, will reveal that $r(z;-1,m)$ provides global exponential convergence up to $m = 4$ while $r(z;\gamma_{\rm opt}, m)$ may only be used for the same purpose up to $m = 3$. Since, $\abs{r(z;-1,4)} \le \abs{r(z;\gamma_{\rm opt},3)}$, our focus for the remainder of the paper is on $r(z;-1,4)$, the conformal map with the smallest modulus.

\begin{figure}[ht]
\begin{center}
\includegraphics[width=\textwidth]{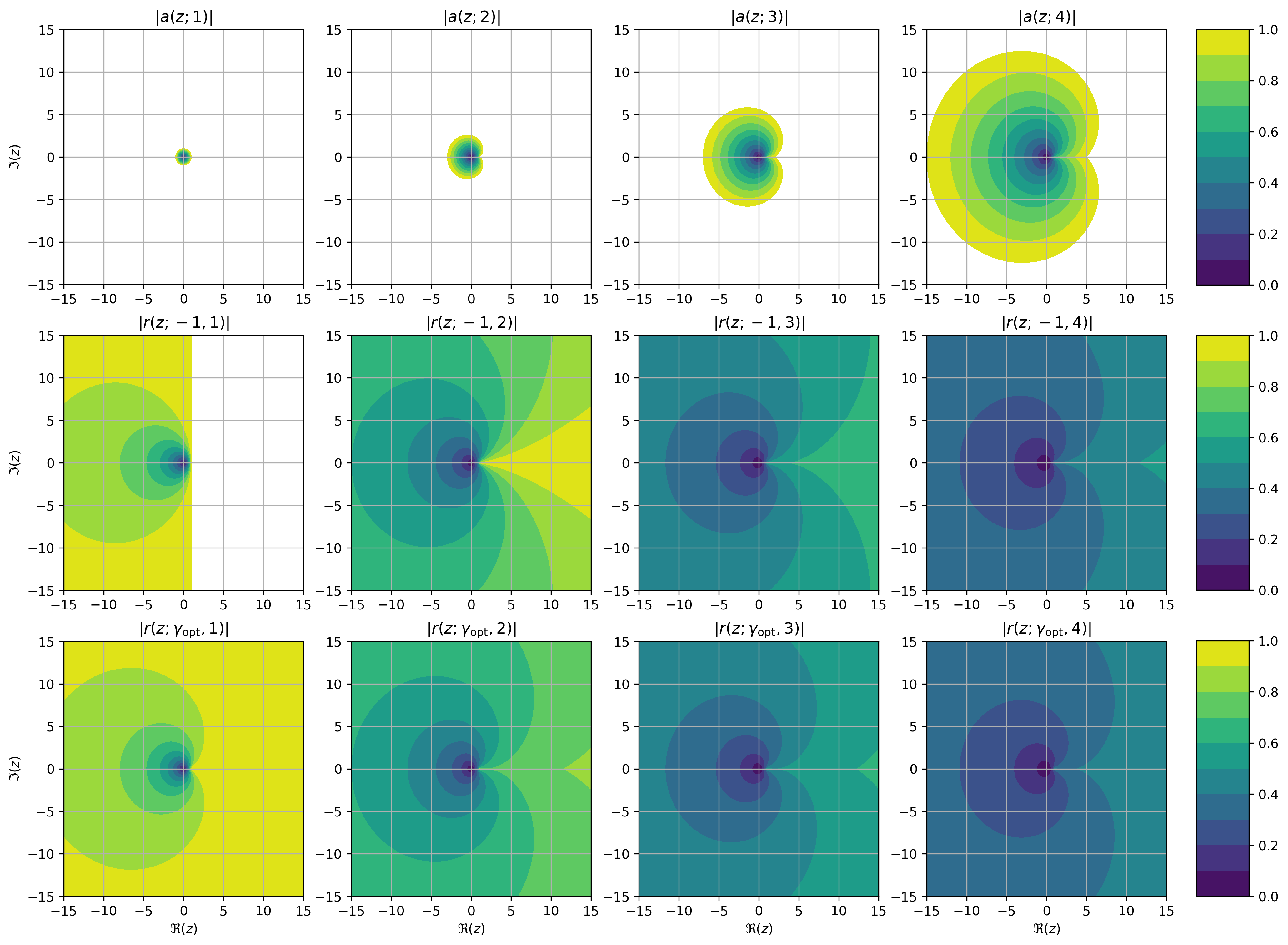}
\end{center}
\caption{Contour plot of moduli of the conformal maps $a(z;m)$, $r(z;-1,m)$ and $r(z;\gamma_{\rm opt},m)$ for $m=1,2,3,4$. Level sets range from $0$ to $1$ in steps of $1/10$, illustrating the images of concentric circles in the $w$-plane. While $r(z;\gamma_{\rm opt}, 4)$ is pointwise the smallest in modulus, its use has been ruled out in \S~\ref{section:rationalconformalmap}, so the best transformation is in fact $r(z;-1,4)$.}
\label{fig:contourf_conformalmaps}
\end{figure}

%%%%%%%%%%%%%%%%%%%%%%%%%%%%%%%%%%%%%%%%%%%%%%%%%%%%%%%%%%%%%%%%%%%%%%%%%%%%%%%%%%%%%%%%%%%%%%%%%%%%%%%%%%%%%%%%
%%%%%%%%%%%%%%%%%%%%%%%%%%%%%%%%%%%%%%%%%%%%%%%%%%%%%%%%%%%%%%%%%%%%%%%%%%%%%%%%%%%%%%%%%%%%%%%%%%%%%%%%%%%%%%%%
%%%%%%%%%%%%%%%%%%%%%%%%%%%%%%%%%%%%%%%%%%%%%%%%%%%%%%%%%%%%%%%%%%%%%%%%%%%%%%%%%%%%%%%%%%%%%%%%%%%%%%%%%%%%%%%%
%%%%%%%%%%%%%%%%%%%%%%%%%%%%%%%%%%%%%%%%%%%%%%%%%%%%%%%%%%%%%%%%%%%%%%%%%%%%%%%%%%%%%%%%%%%%%%%%%%%%%%%%%%%%%%%%

\subsection{The transformed generalized hypergeometric differential equation}

With the rational-algebraic conformal map in Eq.~\eqref{eq:RAT} with $\gamma=-1$, we now describe its transformation of the generalized hypergeometric differential equation. Letting:
\[
u(w):=
\pFq{q+1}{q}\left(
\begin{array}{c}
\alpha_1,\ldots,\alpha_{q+1}\\
\beta_1,\ldots,\beta_q
\end{array};
z=1-\left(\frac{1-w}{1+w}\right)^m
\right),
\]
we will find a recurrence relation for the coefficients in the Maclaurin series:
\[
u(w) = \sum_{n=0}^\infty u_n w^n.
\]
By the chain rule:
\[
\frac{\ud}{\ud z} = \frac{1}{2m}\frac{(1+w)^{m+1}}{(1-w)^{m-1}}\frac{\ud}{\ud w},
\]
and if:
\begin{equation}
p(w;m) = (1+w)^m - (1-w)^m,
\end{equation}
then:
\[
z\frac{\ud}{\ud z} = \left[1-\left(\frac{1-w}{1+w}\right)^m\right]\frac{1}{2m}\frac{(1+w)^{m+1}}{(1-w)^{m-1}}\frac{\ud}{\ud w} = \frac{p(w;m)}{2m}\frac{1+w}{(1-w)^{m-1}}\frac{\ud}{\ud w}.
\]
Hence, the differential equation \eqref{eq:pFqODE} for $\pFq{q+1}{q}$ in the transformed coordinates becomes:
\begin{align}
u(0) & = 1,\nonumber\\
\frac{(1+w)^{m+1}}{(1-w)^{m-1}} \frac{\ud}{\ud w} \prod_{i=1}^q\left[\frac{p(w;m)(1+w)}{(1-w)^{m-1}}\frac{\ud}{\ud w} + \tilde{\beta}_i\right] u(w) & = \prod_{i=1}^{q+1}\left[\frac{p(w;m)(1+w)}{(1-w)^{m-1}}\frac{\ud}{\ud w} + \tilde{\alpha}_i\right] u(w).\label{eq:RATODE}
\end{align}
Here, we have used $\tilde{\beta}_k = 2m(\beta_k-1)$ and $\tilde{\alpha}_k = 2m\alpha_k$ to keep the expression compact.

The Maclaurin series provides analytic continuation to the whole cut complex plane so long as any new singular points of the transformed differential equation are not inside the unit disk. Eq.~\eqref{eq:RATODE} has readily identifiable regular singular points at $w = \pm1$, and the remaining singular points come from an analysis of $p(w;m)$. It is confirmed that the roots of $p(w;m)$ include $w=0$ and are otherwise purely imaginary and are given by $\ii\tan(k\pi/m)$ for $k=-\lfloor\frac{m-1}{2}\rfloor,\ldots,\lfloor\frac{m-1}{2}\rfloor$. For $m=4$ and $k=1$, the smallest complex conjugate pair is at $\pm\ii\tan(\pi/4) = \pm\ii$. Thus, for $m>4$, the smallest conjugate pair will be inside the unit disk. As the roots of $p(w;m)$ are all distinct and also distinct from the other two at $\pm1$, we conclude that all singular points are regular. In summary, while the algebraic transformation of Jacobs, Piret, and Fornberg~\cite{Jacobs-Piret-Fornberg-1-26} has the practical restriction $1\le m\le 6$, the rational-algebraic transformation is practically restricted to $1\le m\le 4$.

For any other value of $\gamma$, including $\gamma_{\rm opt}$, in the transformation $w = r(z;\gamma,m)$, the associated polynomial $p(w;\gamma,m) = (1-\gamma w)^m-(1-w)^m$ has roots:
\[
w_{k,m} = \frac{1-e^{2\pi\ii k/m}}{1-\gamma e^{2\pi\ii k/m}},\quad{\rm for}\quad k = 0,\ldots,m-1.
\]
Since $\arg(\gamma_{\rm opt})\in (\pi-\frac{\pi}{m}, \pi+\frac{\pi}{m})$, the smallest root, in modulus, arises on the boundary of the admissible arc on the circle. In this case:
\begin{align*}
\abs{w_{k,2}}_{\gamma = e^{\ii(\pi\pm\frac{\pi}{2})}} & \ge 1,\quad{\rm for}\quad k=0,1,\\
\abs{w_{k,3}}_{\gamma = e^{\ii(\pi\pm\frac{\pi}{3})}} & \ge 1,\quad{\rm for}\quad k=0,1,2,
\end{align*}
while:
\[
\abs{w_{1,4}}_{\gamma = e^{\ii(\pi-\frac{\pi}{4})}} = \abs{\frac{1-\ii}{1-\ii e^{\ii(\pi-\frac{\pi}{4})}}} = \sqrt{\frac{2}{2+\sqrt{2}}} < 1.
\]
Thus the use of $\gamma_{\rm opt}$ is practically restricted to $1\le m\le 3$ while $\gamma=-1$ enjoys the faster global exponential decay provided by the transformation all the way up to $m=4$.

To return to a differential equation with polynomial coefficients, we multiply both sides of~\eqref{eq:RATODE} by $(1-w)^{(q+1)(m-1)+q}$ and use the following identities. 

Let $s = (k-1)(m-1)+j-1$, then
\begin{equation}\label{eq:key_identity_1}
(1-w)^{s+m}
\left[
\frac{p(w;m)(1+w)}{(1-w)^{m-1}}\frac{\ud}{\ud w}
+ \alpha
\right]
=
\left[
p(w;m)(1-w^2)\frac{\ud}{\ud w}
+ s\,p(w;m)(1+w)
+ \alpha(1-w)^m
\right]
(1-w)^s.
\end{equation}
Moreover, for any integer $k \geq 1$,
\begin{equation}\label{eq:key_identity_2}
(1-w)^k \frac{\ud}{\ud w}
=
\left[
(1-w)\frac{\ud}{\ud w} + k-1
\right]
(1-w)^{k-1}.
\end{equation}
Hence, equation~\eqref{eq:RATODE} can be written in the following form:
\begin{equation}\label{Eq: Lw=Rw}   
\mathcal{L}u(w)=\mathcal{R}u(w),
\end{equation}
where the operators $\mathcal{L}$ and $\mathcal{R}$ are defined by
\begin{align*}
\mathcal{L}& := (1-w)^{(q+1)(m-1)+q}\left[\frac{(1+w)^{m+1}}{(1-w)^{m-1}}\frac{\ud}{\ud w}\right]
\prod_{j=1}^{q}\left[\frac{p(w;m)(1+w)}{(1-w)^{m-1}}\frac{\ud}{\ud w}+\tilde{\beta}_j\right]\\
\mathcal{R}&  :=(1-w)^{(q+1)(m-1)+q}\prod_{j=1}^{q+1}\left[\frac{p(w;m)(1+w)}{(1-w)^{m-1}}\frac{\ud}{\ud w} + \tilde{\alpha}_j\right].
\end{align*}
Here, using the identitites~\eqref{eq:key_identity_1} and~\eqref{eq:key_identity_2}, these operators can be written as the products:
\begin{align*}
\mathcal{L} =\mathcal{L}_0\mathcal{L}_1\cdots\mathcal{L}_q,\quad \mathcal{R} &= \mathcal{R}_1\mathcal{R}_2\cdots\mathcal{R}_{q+1},
\end{align*}
with first-order linear differential operators with polynomial coefficients:
\begin{align*}
\mathcal{L}_0:=&\,(1-w^2)(1+w)^m\frac{\ud}{\ud w}
+
\big(qm-1\big)(1+w)^{m+1}, \\
\mathcal{L}_j:=& \,p(w;m)(1-w^2)\frac{\ud}{\ud w}
+
\big((q-j)m-1\big) p(w;m)(1+w)
+
\tilde{\beta}_j (1-w)^m, \quad 1\le j \le q-1, \\
\mathcal{L}_{q}:=&\,p(w;m)(1+w)\frac{\ud}{\ud w}+\tilde{\beta}_q(1-w)^{m-1},\\
\mathcal{R}_j:=&\, p(w;m)(1-w^2)\frac{\ud}{\ud w}+[(q-j+1)m-1](1+w)p(w;m)+\tilde{\alpha}_{j}(1-w)^m,  \quad 1\le j \le q, \\
\mathcal{R}_{q+1}:=&\,p(w;m)(1+w)\frac{\ud}{\ud w}+\tilde{\alpha}_{q+1}(1-w)^{m-1}.
\end{align*}
It is clear that the actions of ${\cal L}$ and ${\cal R}$ on Maclaurin series can be represented by infinite banded matrices $L$ and $R$ such that:
\begin{equation}\label{eq:Lu=Ru}
{\cal L} \begin{pmatrix} 1 & w & w^2 & \cdots\end{pmatrix}\begin{pmatrix} u_0\\u_1\\u_2\\\vdots\end{pmatrix} = \begin{pmatrix} 1 & w & w^2 & \cdots\end{pmatrix}L\bs{u} = \begin{pmatrix} 1 & w & w^2 & \cdots\end{pmatrix}R\bs{u} = {\cal R} \begin{pmatrix} 1 & w & w^2 & \cdots\end{pmatrix}\begin{pmatrix} u_0\\u_1\\u_2\\\vdots\end{pmatrix}.
\end{equation}
We ultimately approximate $u$ by its $n$-term Maclaurin truncation,
\[
u_n(w)
=
\sum_{k=0}^{n-1}u_kw^k,
\]
which can be realized by a judicious use of the canonical orthogonal projection:
\[
P_n=(I_n,\bs{0})\in\R^{n\times\infty}.
\]
In this case, a finite-dimensional linear system for the coefficients is obtained by truncating the linear system \eqref{eq:Lu=Ru}, and using the initial condition $u_0=1$:
\begin{equation}\label{eq. Linear system expanding}
\begin{pmatrix} 1 & 0 & \cdots & 0 \\ \multicolumn{4}{c}{\raisebox{0.5\normalbaselineskip}{$P_{n-1}(L-R)P_{n}^\top$}} \end{pmatrix} \begin{pmatrix} u_0\\ u_1\\ \vdots\\ u_{n-1} \end{pmatrix} = \begin{pmatrix} 1\\ 0\\ \vdots\\ 0 \end{pmatrix}. 
\end{equation}
The polynomial coefficients of the first-order factors $\mathcal{L}_j$ and $\mathcal{R}_j$ imply that their Maclaurin discretization matrices are banded. Indeed, if
\[
\mathcal{T}=
a(w)\frac{\ud}{\ud w}+b(w), \qquad
a(w)=\sum_{j=0}^{d_a}a_jw^j, \qquad
b(w)=\sum_{j=0}^{d_b}b_jw^j,
\]
then
\[
\mathcal{T}w^k=
k\sum_{j=0}^{d_a}a_jw^{k+j-1}+
\sum_{j=0}^{d_b}b_jw^{k+j}.
\]
Thus, the discretization matrix of ${\cal T}$ has at most one superdiagonal and at most $\max\{d_a-1,d_b\}$ subdiagonals, and if $a(0) = 0$, then the matrix is lower-triangular.

Since the polynomial coefficients of each first-order factor $\mathcal{L}_j$ and $\mathcal{R}_j$ have degree $\mathcal{O}(m)$, the corresponding discretization matrices have lower bandwidth $\mathcal{O}(m)$. Since only ${\cal L}_0$ has a leading coefficient that is nonzero at the origin, it is the only factor with a single super-diagonal. Moreover, since bandwidths add under matrix multiplication, both products of $q+1$ such factors has lower bandwidth $\mathcal{O}(qm)$. Then, the matrix $L-R$, and hence $P_{n-1}(L-R)P_n^\top$, has lower bandwidth $\mathcal{O}(qm)$ and upper bandwidth $1$. By vertical concatenation with the row representing the initial condition, the linear system in equation~\eqref{eq. Linear system expanding} is lower-triangular and can be solved by forward substitution in $\mathcal{O}(qmn)$ flops. However, explicitly assembling the banded matrices $L$ and $R$, by successive banded matrix multiplication requires $\mathcal{O}(q^2m^2n)$ flops and ${\cal O}(qmn)$ storage. This assembly cost motivates the approach developed in \S~\ref{section:compoundrecurrencerelation}, where the factored differential equation \eqref{Eq: Lw=Rw factored} is converted into a compound recurrence relation for the coefficients. This formulation preserves the factored structure and allows the coefficients to be computed in $\mathcal{O}(qmn)$ flops and ${\cal O}(qm)$ storage, without explicitly forming the truncated system.

%%%%%%%%%%%%%%%%%%%%%%%%%%%%%%%%%%%%%%%%%%%%%%%%%%%%%%%%%%%%%%%%%%%%%%%%%%%%%%%%%%%%%%%%%%%%%%%%%%%%%%%%%%%%%%%%
%%%%%%%%%%%%%%%%%%%%%%%%%%%%%%%%%%%%%%%%%%%%%%%%%%%%%%%%%%%%%%%%%%%%%%%%%%%%%%%%%%%%%%%%%%%%%%%%%%%%%%%%%%%%%%%%
%%%%%%%%%%%%%%%%%%%%%%%%%%%%%%%%%%%%%%%%%%%%%%%%%%%%%%%%%%%%%%%%%%%%%%%%%%%%%%%%%%%%%%%%%%%%%%%%%%%%%%%%%%%%%%%%
%%%%%%%%%%%%%%%%%%%%%%%%%%%%%%%%%%%%%%%%%%%%%%%%%%%%%%%%%%%%%%%%%%%%%%%%%%%%%%%%%%%%%%%%%%%%%%%%%%%%%%%%%%%%%%%%

\subsection{Compound recurrence relation}\label{section:compoundrecurrencerelation}

Recall the series expansion:
\[
u(w)=\sum_{n=0}^\infty u_nw^n,
\]
For a sequence $\{u_n\}$, the forward and backward shift operators are defined by
\[
\EE u_n = u_{n+1}, \qquad \EE^{-1} u_n = u_{n-1}.
\]
In particular,
\begin{equation}\label{eq. shift prop 1}
w^k u(w) \longleftrightarrow \EE^{-k}\, u_n,
\end{equation}
and
\begin{equation}\label{eq. shift prop 2}
\frac{\ud}{\ud w} u(w) \longleftrightarrow (n+1)\EE\, u_n.
\end{equation}
Now, recall the factorized differential equation:
\begin{equation}\label{Eq: Lw=Rw factored}   
\mathcal{L}_0\mathcal{L}_1\cdots\mathcal{L}_q u(w)=\mathcal{R}_1\mathcal{R}_2\cdots\mathcal{R}_{q+1} u(w).
\end{equation}
Using the relations~\eqref{eq. shift prop 1} and~\eqref{eq. shift prop 2}, for each factor in~\eqref{Eq: Lw=Rw factored}, yields a factorized recurrence relation of the form:
\begin{equation}\label{Eq: Recurrence L= R}
    L_0^{(m)}L_1^{(m)}\cdots L_q^{(m)}\, u_n
    =
    R_1^{(m)}R_2^{(m)}\cdots R_{q+1}^{(m)}\, u_n.
\end{equation}
The explicit formulae for the factors in \eqref{Eq: Recurrence L= R}, obtained after collecting like shifts, are given below. We adopt the convention that $\binom{m}{r}=0$ for $r<0$ or $r>m$, and define
\begin{equation*}
    \chi_o(k)=\frac{1-(-1)^k}{2}
    =
    \begin{cases}
        1, & \text{if } k \text{ is odd},\\
        0, & \text{if } k \text{ is even}.
    \end{cases}
\end{equation*}
The terms in the left-hand side are:
\begin{equation}\label{eq: def_L0}
L_0^{(m)}
    =
    \sum_{k=-1}^{m+1}
    \left[
        \left(\binom{m}{k+1}-\binom{m}{k-1}\right)(n-k)
        +(qm-1)\binom{m+1}{k}
    \right]\EE^{-k},
\end{equation}
for $j=1,\dots,q-1$,
\begin{equation*}
\begin{split}
L_j^{(m)}
    =&
    \sum_{k=0}^{m+1}
    \bigg[
        2\chi_o(k+1)\left(\binom{m}{k+1}-\binom{m}{k-1}\right)(n-k) \\
        &\quad + ((q-j)m-1)\left(2\chi_o(k)\binom{m}{k}-2\chi_o(k-1)\binom{m}{k-1}\right)
        + (-1)^k\tilde{\beta}_j\binom{m}{k}
    \bigg] \EE^{-k},
\end{split}
\end{equation*}
and
\begin{equation*}
    L_q^{(m)}
    =
    \sum_{k=0}^{m+1}
    \left[
        2\binom{m}{2\floor{\frac{k}{2}}+1}(n-k)
        +(-1)^k\tilde{\beta}_q\binom{m-1}{k}
    \right]\EE^{-k}.
\end{equation*}
In the right-hand side, for $j=1,\dots,q$,
\begin{equation*}
\begin{split}
R_j^{(m)}
    &= \sum_{k=0}^{m+1}
    \biggl[
        2\chi_o(k+1)\left(\binom{m}{k+1}-\binom{m}{k-1}\right)(n-k) \\
        &\qquad + ((q-j+1)m-1)\left(2\chi_o(k)\binom{m}{k}-2\chi_o(k-1)\binom{m}{k-1}\right)
        + (-1)^k\tilde{\alpha}_j\binom{m}{k}
    \biggr] \EE^{-k},
\end{split}
\end{equation*}
and
\begin{equation*}
    R_{q+1}^{(m)}
    =
    \sum_{k=0}^{m+1}
    \left[
        2\binom{m}{2\floor{\frac{k}{2}}+1}(n-k)
        +(-1)^k\tilde{\alpha}_{q+1}\binom{m-1}{k}
    \right]\EE^{-k}.
\end{equation*}
The recurrence relation in Eq.~\eqref{Eq: Recurrence L= R} is also valid with the initial conditions $u_0=1,\, u_{-1}=\cdots=u_{-m-1}=0$. Thus, given $u_n,u_{n-1},\ldots,u_{n-(m+1)}$, the recurrence can be used to compute the next coefficient $u_{n+1}$. This can be done while preserving the factorized structure of~\eqref{Eq: Recurrence L= R}, which avoids expanding it into a single long recurrence.

We introduce auxiliary sequences associated with the successive application of its factors. On the left-hand side, set
\begin{equation}\label{eq:left auxiliary sequences 0}
z_n^{(0)}=u_n,
\end{equation}
and, for \(r=1,\ldots,q\), define recursively
\begin{equation}\label{eq:left auxiliary sequences}
z_n^{(r)}=L_{q-r+1}^{(m)}z_n^{(r-1)}.
\end{equation}
Thus,
\[
z_n^{(1)}=L_q^{(m)}u_n,\quad
z_n^{(2)}=L_{q-1}^{(m)}L_q^{(m)}u_n,\quad\cdots\quad
z_n^{(q)}=L_1^{(m)}\cdots L_q^{(m)}u_n.
\]
Similarly, on the right-hand side, define
\begin{equation}\label{eq:right auxiliary sequences}
\tilde{z}_n^{(1)}= R_{q+1}^{(m)}u_n,
\end{equation}
and, for \(r=2,\ldots,q\),
\begin{equation}\label{eq:right auxiliary sequences r}
\tilde z_n^{(r)} = R_{q-r+2}^{(m)}\tilde z_n^{(r-1)}.
\end{equation}
Hence,
\[
\widetilde z_n^{(2)}=R_q^{(m)}R_{q+1}^{(m)}u_n,\quad\widetilde z_n^{(3)}=R_{q-1}^{(m)}R_q^{(m)}R_{q+1}^{(m)}u_n,\quad\cdots\quad
\widetilde z_n^{(q)}=R_2^{(m)}\cdots R_{q+1}^{(m)}u_n.
\]
With this notation, the full compound recurrence \eqref{Eq: Recurrence L= R} reduces to
\begin{equation}\label{eq:outer auxiliary recurrence}
L_0^{(m)}z_n^{(q)}= R_1^{(m)}\tilde z_n^{(q)}.
\end{equation}

This representation makes it possible to advance the recurrence one factor at a time. To see this, write
\[
L_0^{(m)}
=
\ell_{0,-1}^{(m)}(n)\EE
+
\sum_{k=0}^{m+1}\ell_{0,k}^{(m)}(n)\EE^{-k},
\]
and, for \(j\geq1\),
\[
L_j^{(m)} =
\sum_{k=0}^{m+1}\ell_{j,k}^{(m)}(n)\EE^{-k},
\qquad
R_j^{(m)} =
\sum_{k=0}^{m+1}r_{j,k}^{(m)}(n)\EE^{-k}.
\]
The factor \(L_0^{(m)}\) is distinguished from all the others by the
presence of the forward shift \(\EE\).  In fact, from its definition in~\eqref{eq: def_L0}, we have
\[
\ell_{0,-1}^{(m)}(n)=n+1.
\]
Hence, if the auxiliary sequences are known up to index $n$, Eq.~\eqref{eq:outer auxiliary recurrence} determines the next value $z_{n+1}^{(q)}$:
\begin{equation}\label{eq:update outer auxiliary}
z_{n+1}^{(q)}= \frac{1}{n+1}
\left[
\sum_{k=0}^{m+1}
r_{1,k}^{(m)}(n)\tilde z_{n-k}^{(q)}-\sum_{k=0}^{m+1}
\ell_{0,k}^{(m)}(n)z_{n-k}^{(q)}
\right].
\end{equation}
Once $z_{n+1}^{(q)}$ has been computed, we proceed backward through the factors on the left-hand side, using Eq.~\eqref{eq:left auxiliary sequences} to compute the remaining auxiliary quantities. For example,
\[
z_{n+1}^{(q)} = L_1^{(m)}z_{n+1}^{(q-1)}
\]
determines $z_{n+1}^{(q-1)}$, since this is the only unknown term in the right-hand side, as  all its preceding values are already available. The same procedure can then be applied to
\[
z_{n+1}^{(q-1)} = L_2^{(m)}z_{n+1}^{(q-2)},
\]
and continued successively through the remaining factors.  In general, if $j=q-r+1$, the relation
\[
z_{n+1}^{(r)}
=
\sum_{k=0}^{m+1}
\ell_{j,k}(n+1)z_{n+1-k}^{(r-1)}
\]
gives
\begin{equation}\label{eq:backward auxiliary update}
z_{n+1}^{(r-1)}=
\frac{z_{n+1}^{(r)}
- \displaystyle\sum_{k=1}^{m+1}
\ell_{j,k}^{(m)}(n+1)z_{n+1-k}^{(r-1)}
}{\ell_{j,0}^{(m)}(n+1)}.
\end{equation}
Proceeding from \(r=q\) down to \(r=1\) eventually yields
\[
z_{n+1}^{(0)}=u_{n+1}.
\]
After \(u_{n+1}\) is known, the auxiliary sequences on the right-hand side can be updated directly using Eqs.~\eqref{eq:right auxiliary sequences} and \eqref{eq:right auxiliary sequences r}, now proceeding in the opposite
direction.  First,
\[
\tilde z_{n+1}^{(1)} =
R_{q+1}^{(m)}u_{n+1},
\]
and then
\[
\tilde z_{n+1}^{(2)} = R_q^{(m)}\tilde z_{n+1}^{(1)},
\qquad\ldots,\qquad
\tilde z_{n+1}^{(q)} = R_2^{(m)}\tilde z_{n+1}^{(q-1)}.
\]
At this point, all auxiliary sequences have been updated to index $n+1$,completing one step of the recurrence and allowing the computation of the subsequent coefficients.

In summary, a complete step of the recurrence consists of the following three stages:
\begin{enumerate}
\item Use relation~\eqref{eq:outer auxiliary recurrence} to compute the next value of the outermost left auxiliary sequence $\{z_{n+1}^{(q)}\}$.

\item Work backwards through the left-hand factors using~\eqref{eq:left auxiliary sequences 0} and \eqref{eq:left auxiliary sequences} to successively compute the remaining left-hand auxiliary values. The final step yields $u_{n+1}$.

\item With the value $u_{n+1}$ in hand, use~\eqref{eq:right auxiliary sequences} and \eqref{eq:right auxiliary sequences r} to advance the right-hand auxiliary sequences.
\end{enumerate}

Since each factor involves only the shifts $1,\EE^{-1},\ldots,\EE^{-(m+1)}$, each update requires only a finite number of previously computed values of the corresponding auxiliary sequence. Thus, the recurrence can be advanced while storing only a fixed number of past values for each auxiliary sequence, regardless of how many coefficients have already been computed.

Therefore, using this procedure the products in \eqref{Eq: Recurrence L= R} need not be expanded into a single long recurrence. Instead, the individual factors are applied successively to the corresponding auxiliary sequences. Since each factor involves ${\cal O}(m)$ terms and there are ${\cal O}(q)$ factors, each new series coefficient can be computed in ${\cal O}(qm)$ flops. Consequently, the first $n$ coefficients can be generated in ${\cal O}(qmn)$ flops.  In comparison, the approach based on the infinite-dimensional matrix representation of the differential operators requires ${\cal O}(q^2m^2n)$ flops to assemble the coefficient matrix of the truncated linear system \eqref{eq. Linear system expanding} by successive banded matrix multiplications.

%%%%%%%%%%%%%%%%%%%%%%%%%%%%%%%%%%%%%%%%%%%%%%%%%%%%%%%%%%%%%%%%%%%%%%%%%%%%%%%%%%%%%%%%%%%%%%%%%%%%%%%%%%%%%%%%
%%%%%%%%%%%%%%%%%%%%%%%%%%%%%%%%%%%%%%%%%%%%%%%%%%%%%%%%%%%%%%%%%%%%%%%%%%%%%%%%%%%%%%%%%%%%%%%%%%%%%%%%%%%%%%%%
%%%%%%%%%%%%%%%%%%%%%%%%%%%%%%%%%%%%%%%%%%%%%%%%%%%%%%%%%%%%%%%%%%%%%%%%%%%%%%%%%%%%%%%%%%%%%%%%%%%%%%%%%%%%%%%%
%%%%%%%%%%%%%%%%%%%%%%%%%%%%%%%%%%%%%%%%%%%%%%%%%%%%%%%%%%%%%%%%%%%%%%%%%%%%%%%%%%%%%%%%%%%%%%%%%%%%%%%%%%%%%%%%

\section{An approximate Chebyshev series solution to the generalized hypergeometric differential equation}\label{section:Chebyshevseries}

Now that we have a globally convergent series for the generalized hypergeometric functions $\pFq{q+1}{q}$, it will be interesting to compare it to another method that converges globally, the ultraspherical spectral method~\cite{Olver-Townsend-55-462-13,Olver-Slevinsky-Townsend-29-573-20} to approximate the solution of the differential equation with a Chebyshev series. Note that this has been used before by Crespo et al.~\cite{Crespo-et-al-84-1-20} for the Gauss hypergeometric function to produce an analytic continuation on the extended real line by stitching together up to five series, one on the interval $[-1/2, 1/2]$ and up to two each on the regions $[1/2, 3/2]$ and $[-\infty,-1/2]\cup[3/2,\infty]$. For analytic continuation throughout the whole complex plane, Crespo et al.~\cite{Crespo-et-al-84-1-20} solve a Laplace problem on ellipses with boundary data determined by the values on the real line. The entire procedure produces a representation of a particular Gauss hypergeometric function on the whole complex plane, but requires a rather large storage scheme and precomputation time if one is just comparing numerical evaluation at a particular point.

Rather than solving for $\pFq{q+1}{q}$ globally on the real line by the analytic continuation of Crespo et al.~\cite{Crespo-et-al-84-1-20}, for $z\ne0$ the linear transformation $M(x) = z\dfrac{1+x}{2}$, $M : [-1,1]\to [0,z]$, views approximation of hypergeometric functions as evaluation of the Chebyshev series solution to:
\begin{align}
u(-1) & = 1,\nonumber\\
\frac{2}{z}\frac{\ud}{\ud x}\prod_{i=1}^q\left[(1+x)\frac{\ud}{\ud x}+\beta_i-1\right]u(x) &= \prod_{i=1}^{q+1}\left[(1+x)\frac{\ud}{\ud x}+\alpha_i\right]u(x),\label{eq:ChebyshevpFqODE}
\end{align}
at the right endpoint. That is, if:
\[
u(x) = \sum_{n=0}^\infty u_n T_n(x) = \begin{pmatrix} T_0(x) & T_1(x) & T_2(x) & \cdots \end{pmatrix}\begin{pmatrix} u_0\\ u_1\\ u_2\\ \vdots\end{pmatrix} = \bs{T}(x) \bs{u},
\]
is a Chebyshev series solution to the differential equation~\eqref{eq:ChebyshevpFqODE}, then:
\[
\pFq{q+1}{q}\left(\begin{array}{c} \alpha_1,\ldots,\alpha_{q+1}\\ \beta_1,\ldots,\beta_q\end{array}; z\right) = u(1) = \sum_{n=0}^\infty u_nT_n(1) = \sum_{n=0}^\infty u_n.
\]
The ultraspherical spectral method realizes the action of linear operators such as multiplication by $x$, differentiation and conversion between ultraspherical polynomials with integer parameter difference acting on the quasimatrices of Chebyshev polynomials as infinite banded matrices acting on the Chebyshev coefficients. Using results in Appendix~\ref{appendix:ultrasphericalspectralmethod}, the differential equation encodes an infinite-dimensional banded-plus-rank-one linear system for the Chebyshev coefficients:
\begin{align*}
\begin{pmatrix} 1 & -1 & 1 & -1 & \cdots\end{pmatrix}\bs{u} & = 1,\\
\bs{C}^{(q+1)}(x)\left\{\frac{2}{z}D_{(q)}^{(q+1)}\left[(I+X_{(q)})D_{(q-1)}^{(q)} + (\beta_q-1)R_{(q-1)}^{(q)}\right] \cdots \left[(I+X_{(1)})D_T^U + (\beta_1-1)R_T^U\right]\right. &\\
\left. -\left[(I+X_{(q+1)})D_{(q)}^{(q+1)} + \alpha_{q+1}R_{(q)}^{(q+1)}\right] \cdots \left[(I+X_{(1)})D_T^U + \alpha_1R_T^U\right] \right\}\bs{u} & = 0.
\end{align*}
An important property of the ultraspherical spectral method is that the coefficients it returns produce an approximation that is but a constant multiple worse than the true truncated Chebyshev series. More precisely, let $\bs{u}_n$ denote the infinite-dimensional vector whose first $n$ entries are the coefficients computed by the ultraspherical spectral method and whose remaining entries are zero. Using the Sobolev norm:
\[
\|\bs{u}\|_{\ell_\lambda^2}^2
=
\sum_{k=0}^{\infty}
|u_k|^2(k+1)^{2\lambda},
\]
Olver and Townsend~\cite[Theorem 4.5]{Olver-Townsend-55-462-13} show that there is a $C>0$ such that:
\[
\|\bs{u}-P_n^\top\bs{u}_n\|_{\ell_{q+1}^2}
\le
C\,\|\bs{u}-P_n^\top P_n\bs{u}\|_{\ell_{q+1}^2}
\to 0,\quad{\rm as}\quad n\to\infty.
\]

This differential equation uses $M$ to linearly map the domain of orthogonality of the Chebyshev polynomials to the line segment connecting the origin to the evaluation point $z\in\C\setminus[1,\infty)$. As it is known that $\pFq{q+1}{q}$ is analytic on this same cut plane and as the line segment never crosses the cut, we are approximating the Chebyshev series to an analytic function.
\begin{definition}
The Bernstein ellipse is given parametrically by:
\[
{\cal E}_\rho = \left\{ z\in\C: z = \frac{\rho e^{\ii\theta} + \rho^{-1} e^{-\ii\theta}}{2}, \theta\in[0,2\pi)\right\}.
\]
It has foci at $\pm1$ and the parameter $\rho\ge1$ is the sum of its semi-major and semi-minor axes.
\end{definition}

\begin{theorem}
Let a function $u$ analytic on $[-1,1]$ be analytically continuable to the open Bernstein ellipse ${\cal E}_\rho$, where it satisfies $\abs{u(x)} \le M$ for some $0<M<\infty$. Then its Chebyshev series satisfy:
\[
\norm{u-u_n}_{L^2_w(-1,1)} = \norm{u-u_n}_{\ell^2_0} \le \frac{2M \rho^{-n}}{\sqrt{\rho^2-1}}.
\]
\end{theorem}
\begin{proof}
By~\cite[Theorem 8.1]{Trefethen-12}, the Chebyshev coefficients satisfy $\abs{u_0}\le M$ and $\abs{u_k} \le 2M\rho^{-k}$ for $k>0$. Then:
\[
\norm{u-u_n}_{L^2_w(-1,1)} = \norm{u-u_n}_{\ell^2_0} \le 2M\sqrt{\sum_{k=n+1}^\infty\rho^{-2k}}.
\]
\end{proof}
Trefethen's~\cite[Theorem 8.2]{Trefethen-12} gives the corresponding exponential convergence in the uniform norm. Since $u$ is analytic, exponential convergence may be observed in any number of norms, including the Sobolev norms $\ell_\lambda^2$, with differences appearing in the constants rather than the exponential convergence rates.

For this problem, the left endpoint, $x=-1$, is the regular singular point corresponding to $z=0$ in the original equation. The other regular singular point is at $2 - z(1+x) = 0$, or $z(1+x) = 2$, or $x = 2/z-1$. Now, to determine the Bernstein ellipse parameter, $\rho$, the sum of semi-major and semi-minor axes is equal to the sum of the distances of the singular point to the foci, $x=\pm1$. Thus:
\[
\rho^2 - 2s\rho+1 = 0,
\]
where $s = |1/z| + |1/z-1|$.

\begin{theorem}\label{theorem:theoreticalconvergencerate}
The Chebyshev series solution converges at the same exponential rate as the Maclaurin series solution after Schwarz--Christoffel mapping with prevertices at $\pm1$.
\end{theorem}
\begin{proof}
We need to show that $\rho = \abs{w}^{-1} = \abs{r(z;-1,2)}^{-1}$. Since:
\[
\abs{w}^{-1} = \abs{\frac{1+\sqrt{1-z}}{1-\sqrt{1-z}}} = \frac{\abs{(1+\sqrt{1-z})^2}}{\abs{z}} = \frac{1+\abs{1-z} + \sqrt{1-z} + \overline{\sqrt{1-z}}}{\abs{z}},
\]
and since:
\begin{align*}
\rho & = \abs{\frac{1}{z}}+\abs{\frac{1}{z}-1}+\sqrt{\left(\abs{\frac{1}{z}}+\abs{\frac{1}{z}-1}\right)^2-1},\\
& = \frac{1}{\abs{z}}\left[1+\abs{1-z} + \sqrt{(1+\abs{1-z})^2-\abs{z}^2}\right],\\
& = \frac{1}{\abs{z}}\left[1+\abs{1-z} + \sqrt{1-z+\overline{1-z}+2\abs{1-z}}\right],\\
\end{align*}
The claim now follows from the well-known identity:
\[
\sqrt{z}+\overline{\sqrt{z}} = \sqrt{z+\bar{z}+2\abs{z}}.
\]
\end{proof}

We can also analyze the convergence rate of the Chebyshev series solution to the differential equation in the transformed variable $w = r(z;-1,2)$. We use the linear map $M_w:[-1,1]\to [0,w]$, $M_w(x) = w\frac{1+x}{2}$. In this case, the transformed differential equation has two singular points at $w = \pm 1$. Therefore, the Bernstein ellipse parameter is
\[
\rho_w = \min \{\rho_{-1}, \rho_{1}\},
\]
where
\[
\rho_{\pm 1}
=
\left|\frac{1}{w}\right|
+
\left|\frac{1}{w} \pm 1\right|
+
\sqrt{
\left(
\left|\frac{1}{w}\right|
+
\left|\frac{1}{w} \pm 1\right|
\right)^2
-1
}.
\]

Factoring out $\left|\frac{1}{w}\right|$, we obtain
\[
\rho_{\pm 1}
=
\left|\frac{1}{w}\right|
\left(
1
+
|1 \pm w|
+
\sqrt{
\left(1 + |1 \pm w|\right)^2
- |w|^2
}
\right).
\]
This implies that
\[
\left|\frac{1}{w}\right|
\le
\rho_w
\le
4\left|\frac{1}{w}\right|.
\]
The smallest values of $\rho_w$ occur when $w$ approaches the singular points $\pm 1$, while the largest values are obtained when $w$ is close to $0$. Comparing this with Theorem~\ref{theorem:theoreticalconvergencerate}, this estimate indicates that for $w$ near the origin, the Chebyshev series in $w$ converges at most four times faster than the corresponding Chebyshev expansion in the original variable $z$ or the Maclaurin series obtained via the rational transformation.

While Chebyshev series remain a gold standard in approximation theory, the solution of this differential equation is a few orders of magnitude slower than the method we proposed above with the compound recurrence relation for the factored differential equation. Since the linear system is not lower-triangular, we have not discovered an algorithm requiring only ${\cal O}(qm)$ instead of ${\cal O}(qmn)$ storage. On the other hand, that the exponential convergence rate of the Chebyshev is exactly equal to the quadratic-fractional transformation is somewhat surprising and also shows that the rational conformal maps with $m = 3$ and $4$ can even outperform it.

%%%%%%%%%%%%%%%%%%%%%%%%%%%%%%%%%%%%%%%%%%%%%%%%%%%%%%%%%%%%%%%%%%%%%%%%%%%%%%%%%%%%%%%%%%%%%%%%%%%%%%%%%%%%%%%%
%%%%%%%%%%%%%%%%%%%%%%%%%%%%%%%%%%%%%%%%%%%%%%%%%%%%%%%%%%%%%%%%%%%%%%%%%%%%%%%%%%%%%%%%%%%%%%%%%%%%%%%%%%%%%%%%
%%%%%%%%%%%%%%%%%%%%%%%%%%%%%%%%%%%%%%%%%%%%%%%%%%%%%%%%%%%%%%%%%%%%%%%%%%%%%%%%%%%%%%%%%%%%%%%%%%%%%%%%%%%%%%%%
%%%%%%%%%%%%%%%%%%%%%%%%%%%%%%%%%%%%%%%%%%%%%%%%%%%%%%%%%%%%%%%%%%%%%%%%%%%%%%%%%%%%%%%%%%%%%%%%%%%%%%%%%%%%%%%%

\section{A generalized Euler transformation}

The Euler transformation:
\begin{equation}\label{eq:Eulertransformation}
\pFq{2}{1}\left(\begin{array}{c}\alpha_1, \alpha_2\\\beta_1\end{array};z\right) = (1-z)^{\beta_1-\alpha_1-\alpha_2} \pFq{2}{1}\left(\begin{array}{c}\beta_1-\alpha_1, \beta_1-\alpha_2\\\beta_1\end{array};z\right),
\end{equation}
is an under-appreciated transformation that improves the utility of Maclaurin series for the Gauss hypergeometric function. If $\beta_1-\alpha_1-\alpha_2 < 0$, then it is known that the Gauss hypergeometric function is singular at $z=1$. Even in the neighbourhood of this point, a Maclaurin series attempts to approximate a function that is unbounded as $z\to1$ with bounded basis functions, $z^n$. Thus, the only way this is possible is if the coefficients exhibit large algebraic growth. By explicitly factoring the singular part of the function, Eq.~\eqref{eq:Eulertransformation} transforms the series calculation to one with algebraically {\em decaying} coefficients. Conversely, if $\beta_1-\alpha_1-\alpha_2 > 0$, then the Gauss hypergeometric function gets arbitrarily small as $z\to1$. Using basis functions centred at the origin leads to subtractive cancellation in the neighbourhood of $z=1$.

Certain generalizations of the Euler transformation for $\pFq{q+1}{q}$ exist~\cite{Maier-358-39-06,Miller-Paris-62-31-11,Miller-Paris-43-291-13}, though they impose severe parameter restrictions to obtain {\em another} generalized hypergeometric on the right-hand side. In solving the ODE numerically, we need only let $u(z) = (1-z)^\delta v(z)$, where $\delta = \sum_{i}\beta_i -\sum_i\alpha_i$ quantifies the imbalance in parameters, and solve the resulting differential equation for $v(z)$ with the same conformal map of \S~\ref{section:rationalconformalmap}. The resulting differential equation is:
\begin{align}\label{eq:pFqEulerODE}
v(0) = 1, &\nonumber\\
\left[(1-z)\frac{\ud}{\ud z}-(\delta-q)\right]\prod_{i=1}^q\left[(z-z^2)\frac{\ud}{\ud z} + \beta_i - 1-(\beta_i-1+\delta-q+i)z\right] v(z) &\\
 = \prod_{i=1}^{q+1}\left[(z-z^2)\frac{\ud}{\ud z} + \alpha_i-(\alpha_i+\delta+i-1-q)z\right] v(z).&\nonumber
\end{align}
Under the map \eqref{eq:RAT}, and with:
\[
\begin{array}{rcl}
\tilde{\alpha}_k &=& 2m\,\alpha_k, \\[4pt]
\hat{\alpha}_k   &=& 2m\bigl(\alpha_k + \delta - q + k - 1\bigr),
\end{array}
\qquad\qquad
\begin{array}{rcl}
\tilde{\beta}_k &=& 2m(\beta_k - 1), \\[4pt]
\hat{\beta}_k   &=& 2m\bigl(\beta_k - 1 + \delta - q + k\bigr),
\end{array}
\]
the differential equation \eqref{eq:pFqEulerODE} becomes:
\begin{equation}\label{Eq. Euler 1}
\begin{split}
\left[(1-w^2)\frac{\ud}{\ud w}
-2m(\delta-q)
\right]
\prod_{j=1}^q
\left[
\frac{p(w;m)(1-w^2)}
     {(1+w)^m}
\frac{\ud}{\ud w}
+\tilde{\beta}_j
-\hat{\beta}_j
\frac{p(w;m)}{(1+w)^m}
\right]v(w)&\\
=\prod_{j=1}^{q+1}
\left[ \frac{p(w;m)(1-w^2)}
     {(1+w)^m}
\frac{\ud}{\ud w}
+\tilde{\alpha}_j
-\hat{\alpha}_j
\frac{p(w;m)}{(1+w)^m}
\right]v(w).&
\end{split} 
\end{equation}
To obtain a differential equation with polynomial coefficients, in this case we multiply both sides of \eqref{Eq. Euler 1} by \((1+w)^{(q+1)m}\), and use Eq.~\eqref{eq:key_identity_2} and the following modified version of Eq.~\eqref{eq:key_identity_1}:
\begingroup
\small
\begin{equation*}\label{eq:key_identity_3}
(1+w)^{s+m}
\left[
\frac{p(w;m)(1-w^2)}{(1+w)^m}
\frac{\ud}{\ud w}
+\alpha
-\beta\frac{p(w;m)}{(1+w)^m}
\right]
=
\left[
p(w;m)(1-w^2)\frac{\ud}{\ud w}
-p(w;m)\bigl(s(1-w)+\beta\bigr)
+\alpha(1+w)^m
\right]
(1+w)^s.
\end{equation*}
\endgroup
This allows us to write Eq.~\eqref{Eq. Euler 1} as
\begin{equation}\label{Eq: Lw=Rw tilde}   
\widetilde{\mathcal L}\,v(w)=\widetilde{\mathcal{L}}_0 \widetilde{\mathcal{L}}_1\cdots \widetilde{\mathcal{L}}_q \,v(w)=\widetilde{\mathcal{R}}_1 \widetilde{\mathcal{R}}_2 \cdots \widetilde{\mathcal{R}}_{q+1}\, v(w)=\widetilde{\mathcal R}\,v(w).
\end{equation}
where the operators $\widetilde{\mathcal L}$ and $\widetilde{\mathcal R}$ have the following factored form:
\begin{align*}%\label{eq:polynomial_L Euler}
\widetilde{\mathcal L}
&:=
(1+w)^{(q+1)m}
\left[
(1-w^2)\frac{\ud}{\ud w}
-2m(\delta-q)
\right]
\prod_{j=1}^{q}
\left[
\frac{p(w;m)(1-w^2)}
     {(1+w)^m}
\frac{\ud}{\ud w}
+\widetilde{\beta}_j
-\widehat{\beta}_j
\frac{p(w;m)}{(1+w)^m}
\right]
\nonumber\\
&=
(1+w)^m\left[
(1-w^2)\frac{\ud}{\ud w}
-qm(1-w)
-2m(\delta-q)
\right]
\nonumber\\
&\qquad\times
\prod_{j=1}^{q}
\left[
p(w;m)(1-w^2)\frac{\ud}{\ud w}
-p(w;m)
\left(
(q-j)m(1-w)
+\widehat{\beta}_j
\right)
+\widetilde{\beta}_j(1+w)^m
\right]
\end{align*}
and 
\begin{align*}%\label{eq:polynomial_R Euler}
\widetilde{\mathcal R}
&:=
(1+w)^{(q+1)m}
\prod_{j=1}^{q+1}
\left[
\frac{p(w;m)(1-w^2)}
     {(1+w)^m}
\frac{\ud}{\ud w}
+\widetilde{\alpha}_j
-\widehat{\alpha}_j
\frac{p(w;m)}{(1+w)^m}
\right]
\nonumber\\
&=
\prod_{j=1}^{q+1}
\left[
p(w;m)(1-w^2)\frac{\ud}{\ud w}
-p(w;m)
\left(
(q+1-j)m(1-w)
+\widehat{\alpha}_j
\right)
+\widetilde{\alpha}_j(1+w)^m
\right].
\end{align*}
As in Section~\ref{section:compoundrecurrencerelation}, we can represent the differential equation in terms of shift operators, so that Eq.~\eqref{Eq: Lw=Rw tilde} can be rewritten as a factorized recurrence relation of the form:
\begin{equation}\label{Eq: Recurrence L= R Euler}
    \widetilde{L}_0^{(m)}
    \widetilde{L}_1^{(m)}
    \cdots
    \widetilde{L}_{q}^{(m)}\, v_n
    =
    \widetilde{R}_1^{(m)}
    \widetilde{R}_2^{(m)}
    \cdots
    \widetilde{R}_{q+1}^{(n)}\, v_n.
\end{equation}
Here, in the left-hand side,
\begin{equation*}
    \widetilde{L}_{0}^{(m)}    =
    \sum_{k=-1}^{m+2}
    \left[\left(\binom{m}{k+1}-\binom{m}{k-1}\right)(n-k)
    +qm\binom{m}{k-1}
    +m(q-2\delta)\binom{m}{k}
    \right]\EE^{-k},
\end{equation*}
and for $j=1,\dots,q$,
\begin{equation*}
\begin{split}
\widetilde{L}_j^{(m)}
    =&
    \sum_{k=0}^{m+1}
    \bigg[
        2\chi_o(k+1)
        \left(
            \binom{m}{k+1}
            -
            \binom{m}{k-1}
        \right)(n-k)
        \\
        &\qquad  -2\chi_o(k)\binom{m}{k}[(q-j)m+\widehat{\beta}_j]
            +
            2\chi_o(k-1)\binom{m}{k-1}(q-j)m +\widetilde{\beta}_j\binom{m}{k}
    \bigg]\EE^{-k}.
\end{split}
\end{equation*}
On the right-hand side, for $j=1,\dots,q+1$.
\begin{equation*}
\begin{split}
\widetilde{R}_j^{(m)}
    =&
    \sum_{k=0}^{m+1}
    \bigg[
        2\chi_o(k+1)
        \left(
            \binom{m}{k+1}
            -
            \binom{m}{k-1}
        \right)(n-k)
        \\
        &\qquad    -2\chi_o(k)\binom{m}{k}[(q+1-j)m+\widehat{\alpha}_j]   +
            2\chi_o(k-1)\binom{m}{k-1}(q+1-j)m
       +\widetilde{\alpha}_j\binom{m}{k}
    \bigg]\EE^{-k}.
\end{split}
\end{equation*}
Therefore, the generalized Euler transformation preserves the factorized structure required by the recurrence construction of \S~\ref{section:compoundrecurrencerelation}, allowing the coefficients of $v(w)$ to be computed using the same auxiliary-sequence procedure obviating the expansion of the products in \eqref{Eq: Recurrence L= R Euler} into a single long recurrence. Once $v$ has been computed, the original solution is recovered from
\[
u(z)=(1-z)^\delta v(w).
\]
In particular, the transformation retains the $\mathcal O(qmn)$ complexity of the factorized recurrence while removing the algebraic singular behaviour associated with the factor $(1-z)^{\delta}$.

%%%%%%%%%%%%%%%%%%%%%%%%%%%%%%%%%%%%%%%%%%%%%%%%%%%%%%%%%%%%%%%%%%%%%%%%%%%%%%%%%%%%%%%%%%%%%%%%%%%%%%%%%%%%%%%%
%%%%%%%%%%%%%%%%%%%%%%%%%%%%%%%%%%%%%%%%%%%%%%%%%%%%%%%%%%%%%%%%%%%%%%%%%%%%%%%%%%%%%%%%%%%%%%%%%%%%%%%%%%%%%%%%
%%%%%%%%%%%%%%%%%%%%%%%%%%%%%%%%%%%%%%%%%%%%%%%%%%%%%%%%%%%%%%%%%%%%%%%%%%%%%%%%%%%%%%%%%%%%%%%%%%%%%%%%%%%%%%%%
%%%%%%%%%%%%%%%%%%%%%%%%%%%%%%%%%%%%%%%%%%%%%%%%%%%%%%%%%%%%%%%%%%%%%%%%%%%%%%%%%%%%%%%%%%%%%%%%%%%%%%%%%%%%%%%%

\section{Numerical results}

Our numerical experiments are conducted on an iMac Pro (Early 2018) with a $2.3$ GHz Intel Xeon W-2191B with $128$ GB $2.67$ GHz DDR4 RAM. The code is available in Julia in {\tt HypergeometricFunctions.jl}~\cite{Slevinsky-GitHub-HypergeometricFunctionsjl}.

The left panel of Figure~\ref{fig:convergenceandtimes} illustrates the exponential decay rates in successive terms in the approximation of $\pFq{2}{1}((1+\ii, 2-\ii), (3,); e^{\ii \pi/3})$ with the polynomial conformal maps of Jacobs, Piret, and Fornberg~\cite{Jacobs-Piret-Fornberg-1-26}, the rational conformal map of this work described in \S~\ref{section:rationalconformalmap}, and the approximate Chebyshev series returned by square truncations of the ultraspherical spectral method, as described in \S~\ref{section:Chebyshevseries}. Here, we have numerical evidence confirming Theorems~\ref{theorem:Inequalities} and~\ref{theorem:theoreticalconvergencerate}, that: the rational conformal maps converge faster than the corresponding polynomial conformal maps of the same degree, $\abs{r(z;-1,m)} \le \abs{a(z;m)}$; the Schwarz--Christoffel map converges at the same rate as the approximate Chebyshev series; and, that $m=4$ is optimal in the rational conformal map. The right panel of Figure~\ref{fig:convergenceandtimes} illustrates the calculation times of the linear recurrence for computing the Maclaurin series after the rational conformal transformation with $m=4$ averaged over 100 trials. As described in \S~\ref{section:compoundrecurrencerelation}, we have derived a method to leverage the factored differential equation to produce approximations in ${\cal O}(qmn)$ flops. With axes on logarithmic scales, this is evidenced by a doubling, rather than quadrupling, in the calculation time as the number of denominator parameters doubles. The implementation of $\pFq{2}{1}$ is substantially faster than the others because it is a special case for which we optimized the procedure even further. The parameters were chosen as uniformly distributed pseudorandom numbers in $[1,2)$, and the evaluation point is $z=\frac{1}{2}$. The bump in the calculation times around $n = 300$ is due to the fact that $r(\frac{1}{2};-1,4)^{300} \approx 9.8897\times10^{-320}$, and the slowdown is attributable to the use of subnormal numbers. For $n=305$ and beyond, the corresponding value is $0$ in double precision. While choosing a different evaluation point could shift the bump in calculation times into a different decade, it is inevitable for an exponentially convergent algorithm. Nevertheless, it does not affect the asymptotic linear complexity of the approach.

\begin{figure}[htbp]
\begin{center}
\begin{tabular}{cc}
\includegraphics[width=0.486\textwidth]{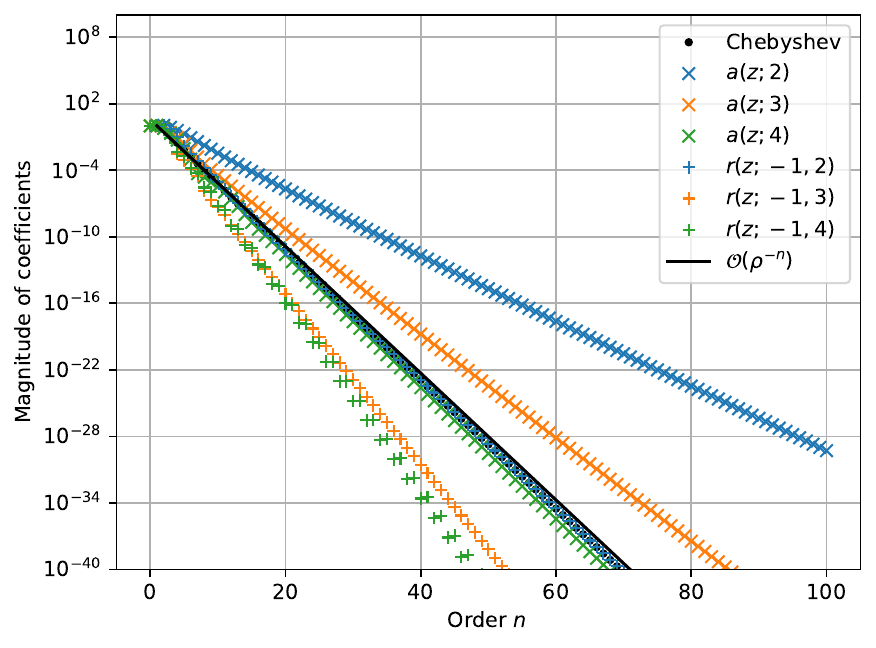} &
\includegraphics[width=0.48\textwidth]{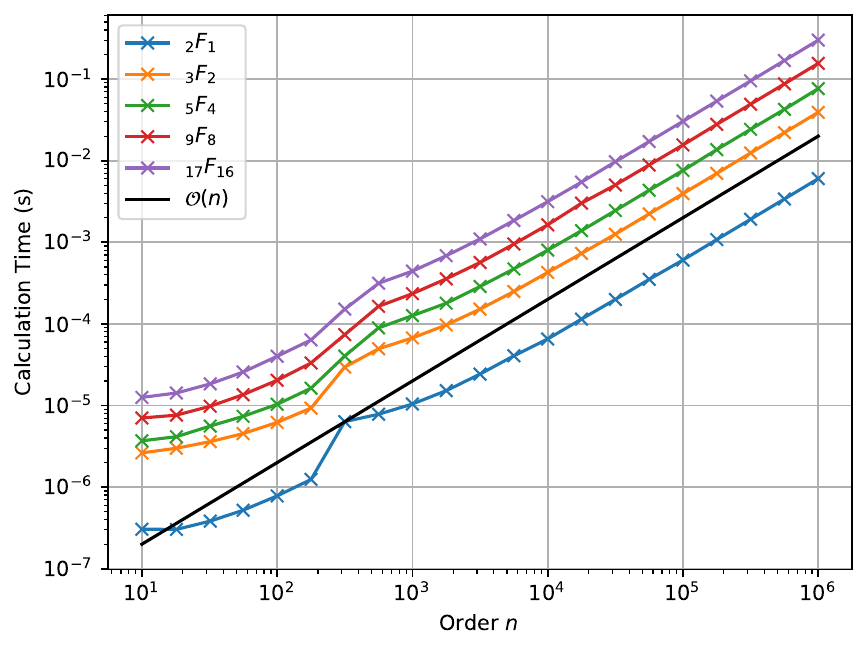}\\
\end{tabular}
\end{center}
\caption{The left panel illustrates the various exponential convergence rates of the method of Jacobs, Piret, and Fornberg~\cite{Jacobs-Piret-Fornberg-1-26}, that of this work described in \S~\ref{section:rationalconformalmap}, and the ultraspherical spectral method, as described in \S~\ref{section:Chebyshevseries}. The right panel illustrates the ${\cal O}(qmn)$ complexity by running the compound linear recurrence relation to a predetermined order $n$.}
\label{fig:convergenceandtimes}
\end{figure}

Figures~\ref{fig:2F1polm2}--\ref{fig:2F1ratm4} illustrate the approximation of the Gauss hypergeometric function $\pFq{2}{1}((1, 7/2), (-7/4,), z)$ on a square subset of the complex plane. The top two panels of each figure are phase portraits, while the bottom two panels are contour plots of the modulus of the relative error. The left panels of each figure illustrate approximation by the $30$-term truncations while the right panels of each figure illustrate approximation until the following stopping criterion:
\[
\abs{u_k(w(z)) - u_{k-1}(w(z))} \le {\tt tol}\abs{u_k(w(z))},
\]
is reached, where ${\tt tol} = 8\epsilon_{\rm mach} = 2^{-49}$. In the bottom right panel of Figures~\ref{fig:2F1polm2}--\ref{fig:2F1polm4}, the level sets demarcating the enlarged theoretical convergence regions are illustrated by black curves. Curiously, in Figure~\ref{fig:2F1polm4} where $m=4$, the Maclaurin series in this polynomial conformal map converges far beyond the theoretical extremal level sets, supporting the observation that the series are also asymptotic (divergent) series beyond these sets. Through these figures, it is clear that the rational transformations are more powerful tools of analytic continuation, and will be even more useful beyond the Gauss hypergeometrics as the analytical rational transformations of \S~\ref{subsection:analyticcontinuationbyrationaltransformations} are more complicated to implement.

\begin{figure}[htbp]
\begin{center}
\includegraphics[width=\textwidth]{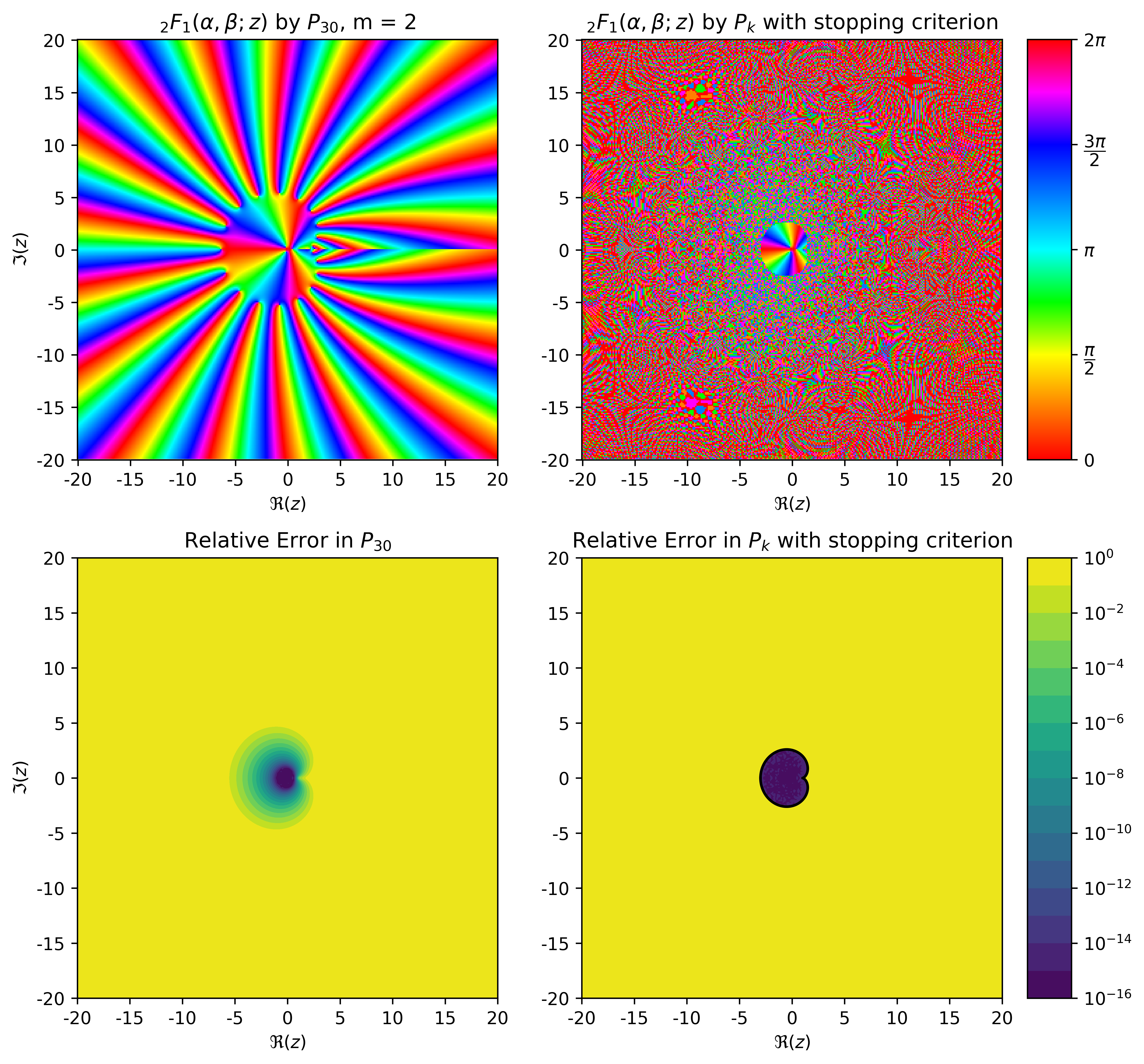}
\end{center}
\caption{Approximation of $\pFq{2}{1}((1, 7/2), (-7/4,), z)$ by Maclaurin series after the polynomial map of Jacobs, Piret, and Fornberg~\cite{Jacobs-Piret-Fornberg-1-26} with $m=2$.}
\label{fig:2F1polm2}
\end{figure}

\begin{figure}[htbp]
\begin{center}
\includegraphics[width=\textwidth]{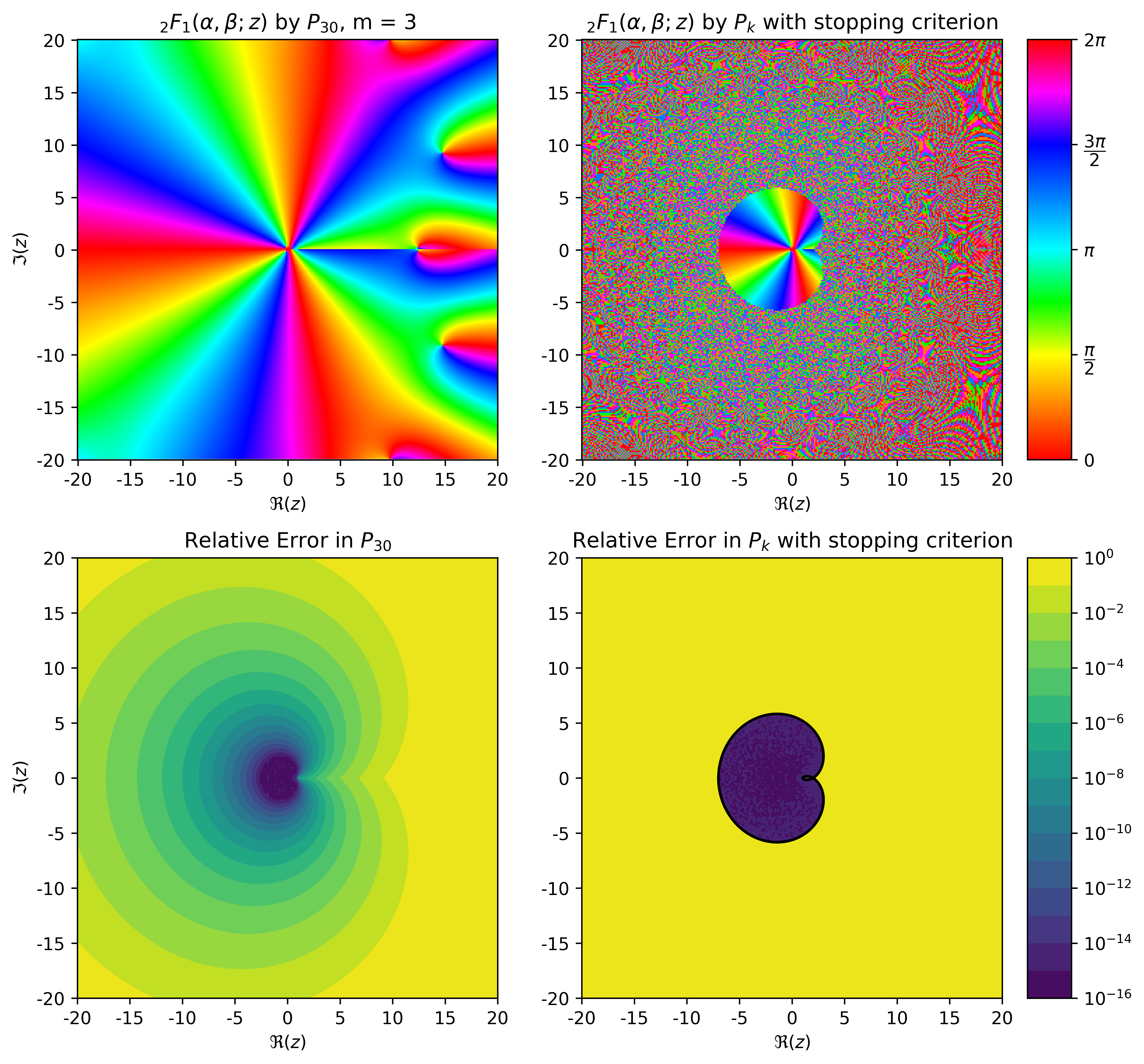}
\end{center}
\caption{Approximation of $\pFq{2}{1}((1, 7/2), (-7/4,), z)$ by Maclaurin series after the polynomial map of Jacobs, Piret, and Fornberg~\cite{Jacobs-Piret-Fornberg-1-26} with $m=3$.}
\label{fig:2F1polm3}
\end{figure}

\begin{figure}[htbp]
\begin{center}
\includegraphics[width=\textwidth]{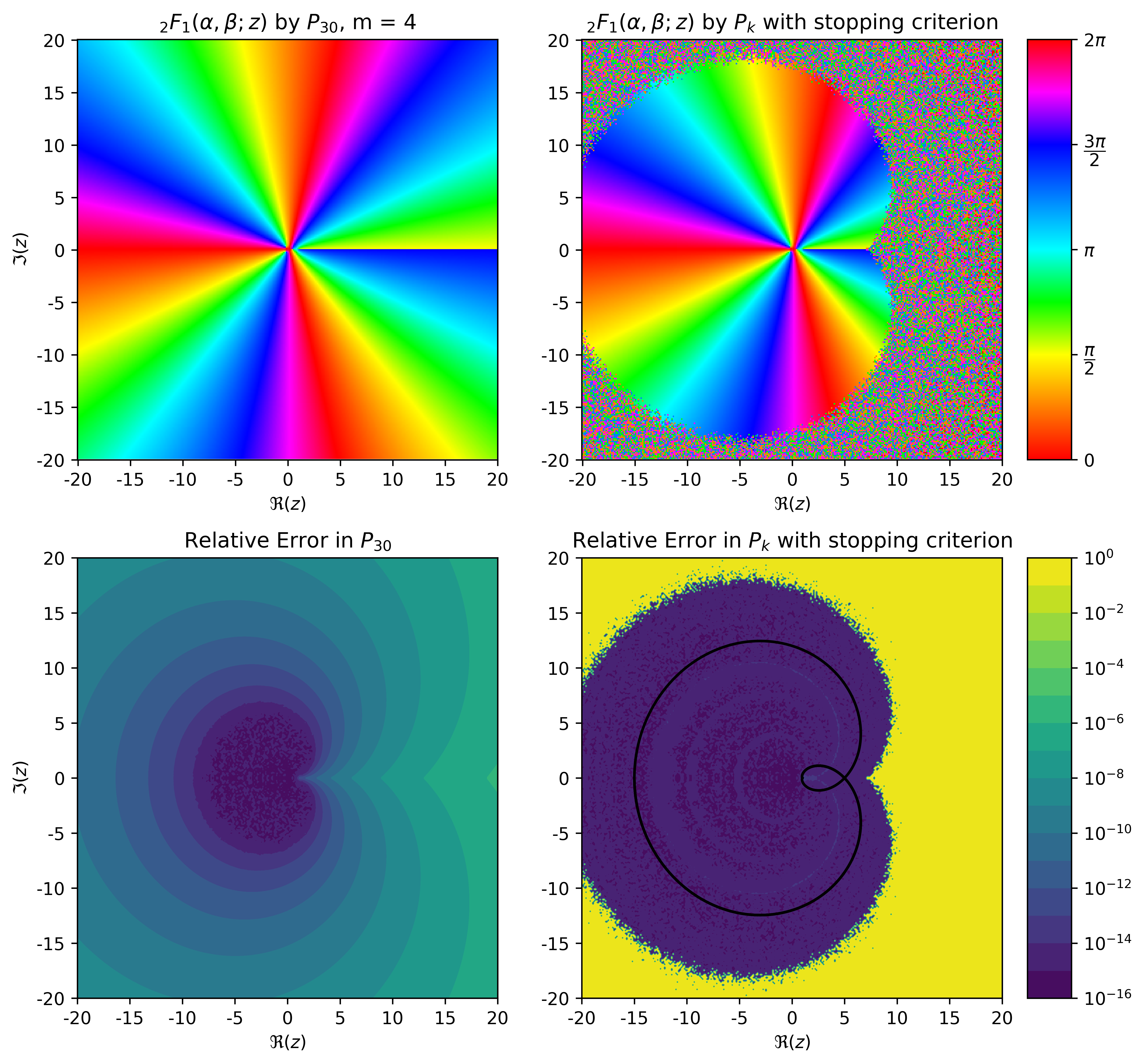}
\end{center}
\caption{Approximation of $\pFq{2}{1}((1, 7/2), (-7/4,), z)$ by Maclaurin series after the polynomial map of Jacobs, Piret, and Fornberg~\cite{Jacobs-Piret-Fornberg-1-26} with $m=4$.}
\label{fig:2F1polm4}
\end{figure}

\begin{figure}[htbp]
\begin{center}
\includegraphics[width=\textwidth]{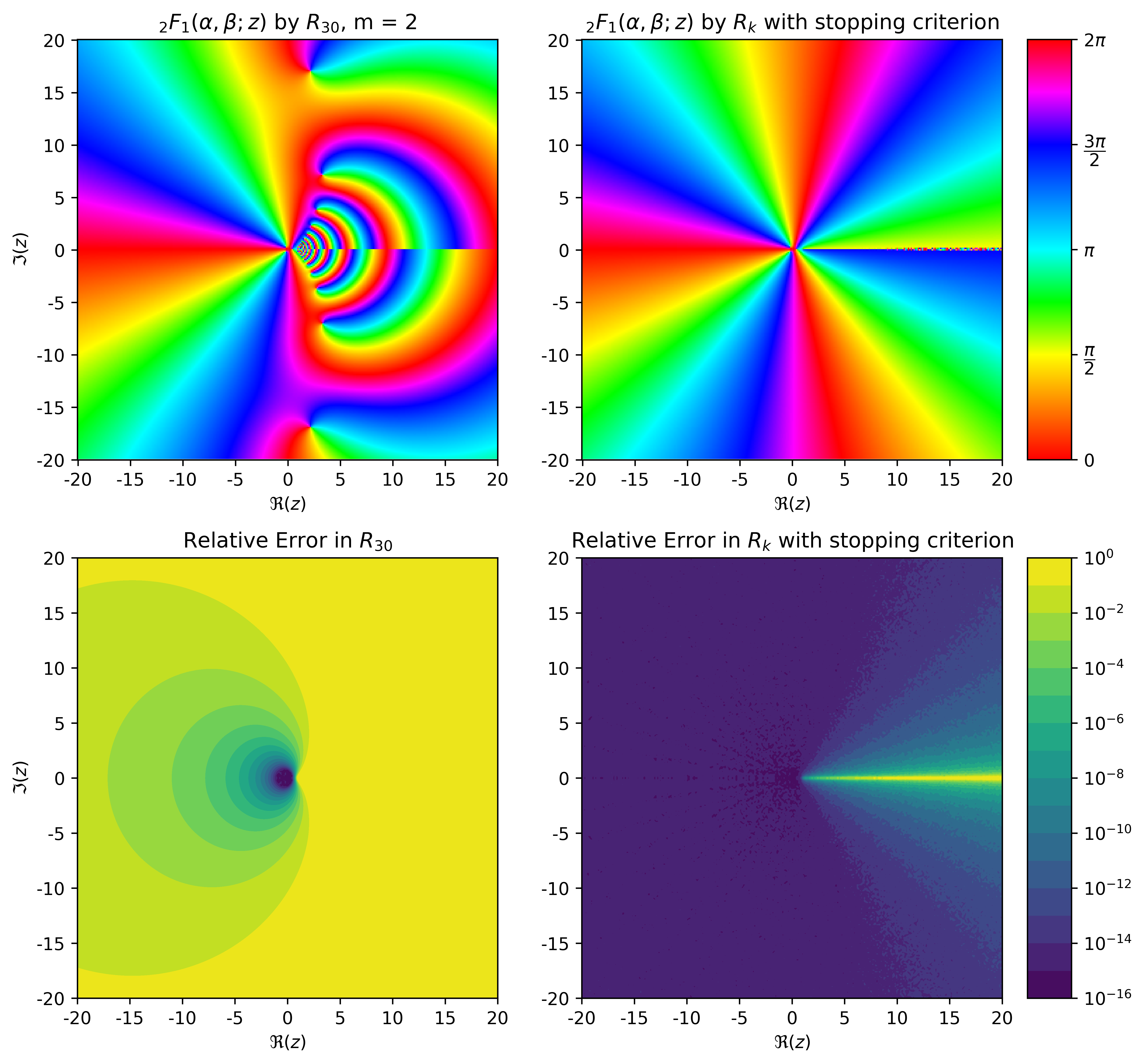}
\end{center}
\caption{Approximation of $\pFq{2}{1}((1, 7/2), (-7/4,), z)$ by Maclaurin series after our rational conformal map with $m=2$.}
\label{fig:2F1ratm2}
\end{figure}

\begin{figure}[htbp]
\begin{center}
\includegraphics[width=\textwidth]{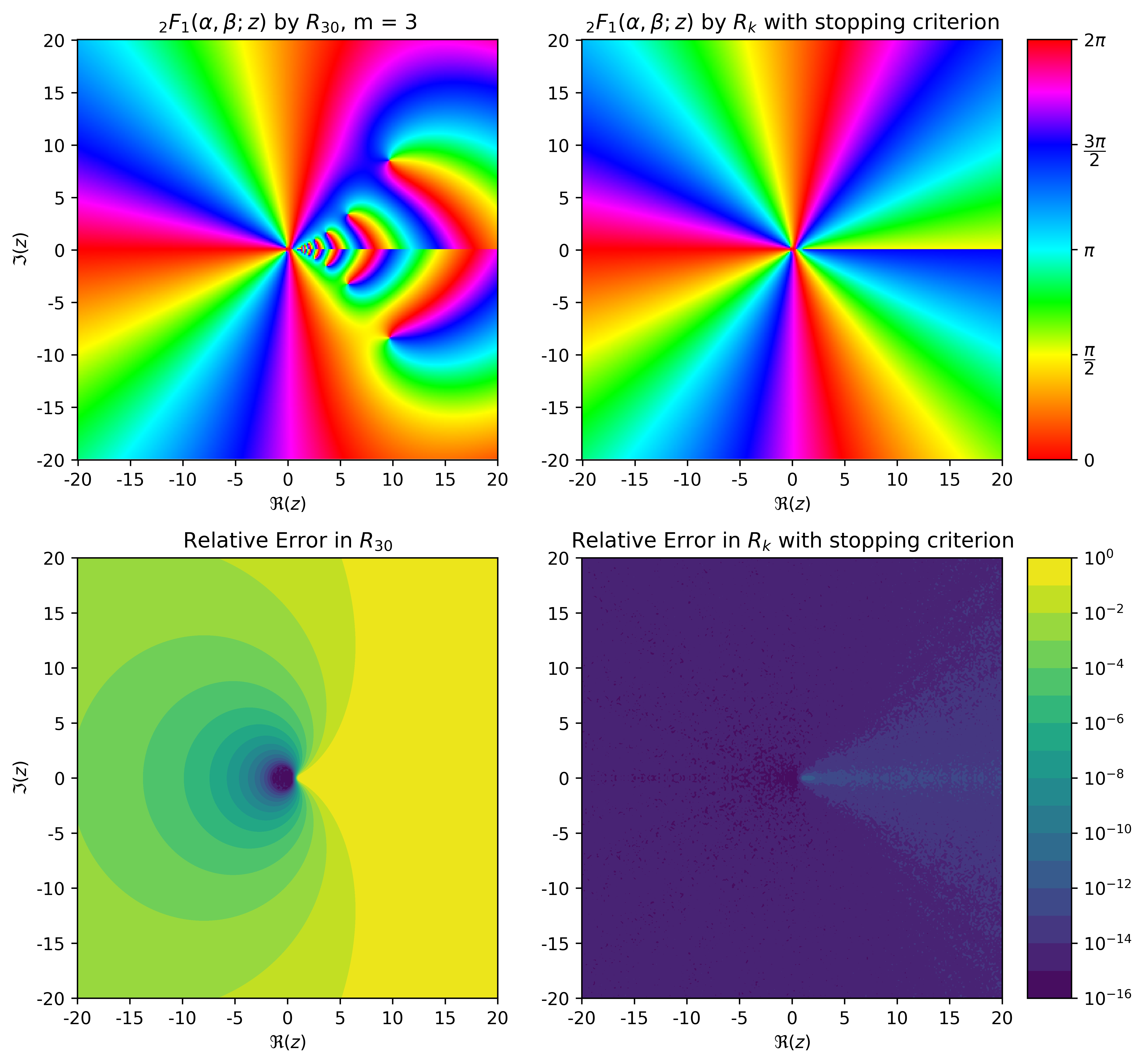}
\end{center}
\caption{Approximation of $\pFq{2}{1}((1, 7/2), (-7/4,), z)$ by Maclaurin series after our rational conformal map with $m=3$.}
\label{fig:2F1ratm3}
\end{figure}

\begin{figure}[htbp]
\begin{center}
\includegraphics[width=\textwidth]{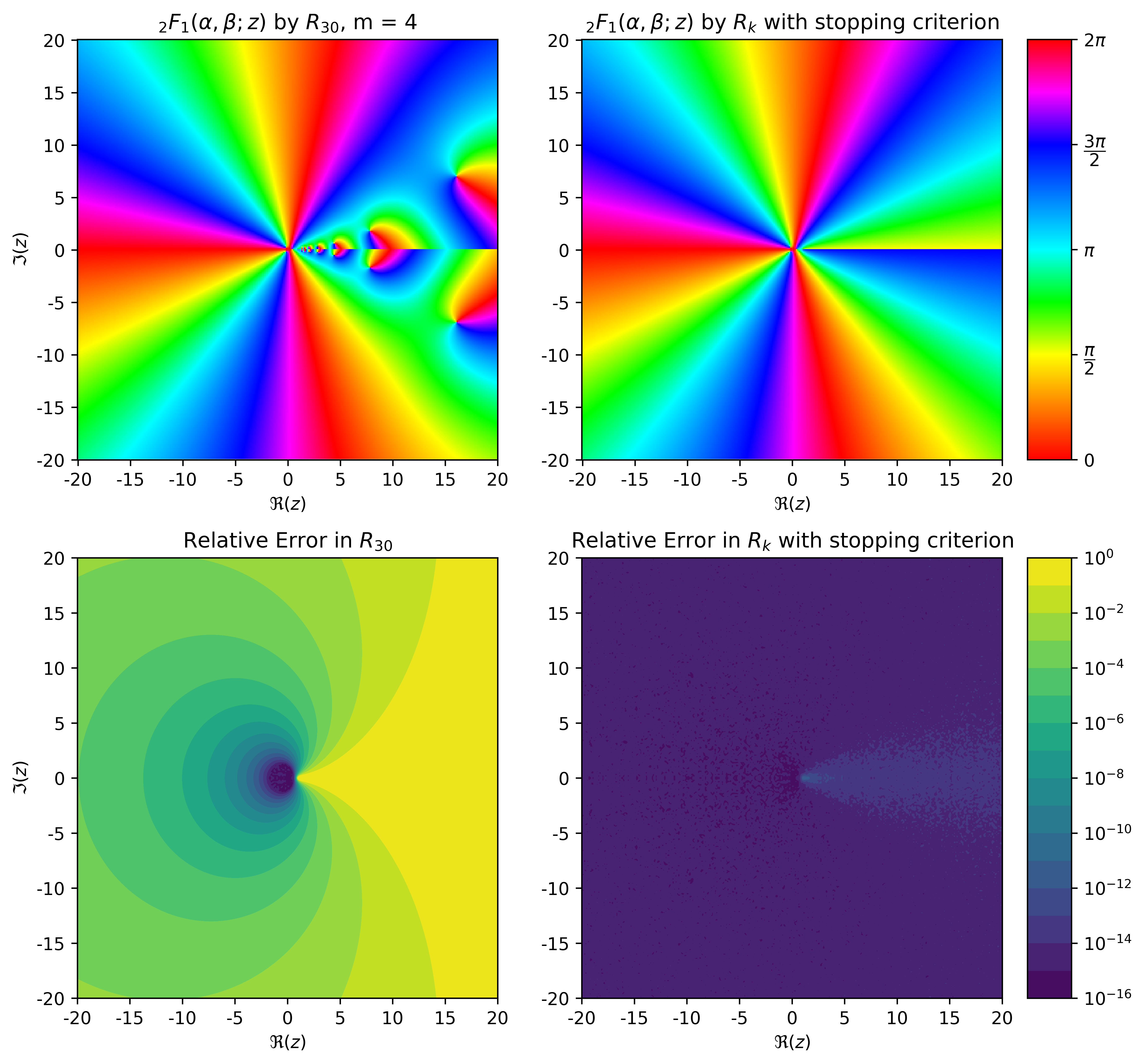}
\end{center}
\caption{Approximation of $\pFq{2}{1}((1, 7/2), (-7/4,), z)$ by Maclaurin series after our rational conformal map with $m=4$.}
\label{fig:2F1ratm4}
\end{figure}

%%%%%%%%%%%%%%%%%%%%%%%%%%%%%%%%%%%%%%%%%%%%%%%%%%%%%%%%%%%%%%%%%%%%%%%%%%%%%%%%%%%%%%%%%%%%%%%%%%%%%%%%%%%%%%%%
%%%%%%%%%%%%%%%%%%%%%%%%%%%%%%%%%%%%%%%%%%%%%%%%%%%%%%%%%%%%%%%%%%%%%%%%%%%%%%%%%%%%%%%%%%%%%%%%%%%%%%%%%%%%%%%%
%%%%%%%%%%%%%%%%%%%%%%%%%%%%%%%%%%%%%%%%%%%%%%%%%%%%%%%%%%%%%%%%%%%%%%%%%%%%%%%%%%%%%%%%%%%%%%%%%%%%%%%%%%%%%%%%
%%%%%%%%%%%%%%%%%%%%%%%%%%%%%%%%%%%%%%%%%%%%%%%%%%%%%%%%%%%%%%%%%%%%%%%%%%%%%%%%%%%%%%%%%%%%%%%%%%%%%%%%%%%%%%%%

\section{Conclusion \& future directions}

The surprising conclusions for analytic continuation of generalized hypergeometric functions to the cut complex plane $\C\setminus[1,\infty)$ are that: Schwarz--Christoffel mapping converges at the same exponential rate as approximate mapped Chebyshev series computed most efficiently by the ultraspherical spectral method; our family of rational conformal maps converges at faster exponential rates than the Schwarz--Christoffel map when we take $m=3$ and $4$; and, the factored form of the generalized hypergeometric differential equation leads to optimal complexities of ${\cal O}(qmn)$ flops and ${\cal O}(qm)$ storage for $n$-term Maclaurin series after transformation beating the product approach of ${\cal O}(q^2m^2n)$ flops and ${\cal O}(qmn)$ storage.

The natural extension is to consider rational mapping tailored to multiply-cut complex planes to solve Fuchsian differential equations with polynomial coefficients in a similar optimal complexity globally throughout the cut complex plane. Applications include Appell~\cite[\S 16.13]{Olver-et-al-NIST-10} and Heun~\cite[\S 31]{Olver-et-al-NIST-10} functions, all of which are defined by differential equations with two cuts in the complex plane. Hale and Tee have produced an optimized routine for Schwarz--Christoffel mapping with multiple infinite complex conjugate vertical slits~\cite{Hale-Tee-31-3195-09}. As with this case, see Eq.~\eqref{eq:SchwarzChristoffelSingleCut}, the integration is performed exactly, which removes the necessity for quadrature in the resolution of the so-called Schwarz--Christoffel parameter problem. Unfortunately, the exact integration for two or more slits reveals a logarithm at the finite slit vertices. These logarithms do not prevent analytic continuation to the cut plane, as their branch cuts are exterior to the unit disk; thus exponential convergence is guaranteed throughout the multiply-slit complex plane. However, the logarithms densify the erstwhile finite-length recurrence relations for the Maclaurin series in the transformed variable. Nevertheless, we are optimistic that rational conformal maps exist, and we explain this as follows. While every biholomorphism from the unit disk to the cut complex plane must be a Schwarz--Christoffel map, we have also shown that other rational conformal maps beat this construction. This apparent contradiction is resolved once it is realized that for $m\ge2$, our $r^{-1}(w;\gamma,m)$ maps {\em a subset} of the unit disk bijectively onto the cut complex plane. Computing such rational maps, however, is the open challenge.

Another avenue of exploration is to derive practical algorithms based on conformal maps for the class of generalized hypergeometric functions that are formally divergent series, $\pFq{p}{q}$ with $p > q+1$. The canonical example is the asymptotic expansion of the exponential integral~\cite[\S 6.2 \& \S 6.6]{Olver-et-al-NIST-10}:
\begin{align*}
\pFq{2}{0}(1,1;z) & = -\frac{e^{-\frac{1}{z}}}{z}E_1\left(-\frac{1}{z}\right) = -\frac{e^{-\frac{1}{z}}}{z} \int_1^\infty \frac{e^{t/z}}{t}\ud t,\\
& = \frac{e^{-\frac{1}{z}}}{z}\left(\gamma-\ln(-z)+\sum_{n=1}^\infty\frac{1}{n n! z^n}\right),\qquad z\in\C\setminus[0,\infty),\\
& \sim \sum_{n=0}^\infty n!z^n\quad{\rm as}\quad z\to 0.
\end{align*}
In this case, the origin is an irregular singular point of the differential equation~\eqref{eq:pFqODE}, manifesting as an essential but removable singularity. While, for $m\ge2$:
\[
w = \frac{1-\sqrt[m]{-z}}{1+\sqrt[m]{-z}},
\]
maps the corresponding cut complex plane into the open unit disk, the initial condition at the origin is now mapped to the {\em boundary} at $w=1$, rather than the center as with $r(z;\gamma,m)$. Further research could uncover how this technique could lead to a promising numerical approach.

\section*{Acknowledgments}

RMS is supported by the Natural Sciences and Engineering Research Council of Canada, through a Discovery Grant (RGPIN-2024-05480). JL is supported by a PIMS postdoctoral research fellowship.

%%%%%%%%%%%%%%%%%%%%%%%%%%%%%%%%%%%%%%%%%%%%%%%%%%%%%%%%%%%%%%%%%%%%%%%%%%%%%%%%%%%%%%%%%%%%%%%%%%%%%%%%%%%%%%%%
%%%%%%%%%%%%%%%%%%%%%%%%%%%%%%%%%%%%%%%%%%%%%%%%%%%%%%%%%%%%%%%%%%%%%%%%%%%%%%%%%%%%%%%%%%%%%%%%%%%%%%%%%%%%%%%%
%%%%%%%%%%%%%%%%%%%%%%%%%%%%%%%%%%%%%%%%%%%%%%%%%%%%%%%%%%%%%%%%%%%%%%%%%%%%%%%%%%%%%%%%%%%%%%%%%%%%%%%%%%%%%%%%
%%%%%%%%%%%%%%%%%%%%%%%%%%%%%%%%%%%%%%%%%%%%%%%%%%%%%%%%%%%%%%%%%%%%%%%%%%%%%%%%%%%%%%%%%%%%%%%%%%%%%%%%%%%%%%%%

\bibliography{/Users/Mikael/Bibliography/Mik}

%%%%%%%%%%%%%%%%%%%%%%%%%%%%%%%%%%%%%%%%%%%%%%%%%%%%%%%%%%%%%%%%%%%%%%%%%%%%%%%%%%%%%%%%%%%%%%%%%%%%%%%%%%%%%%%%
%%%%%%%%%%%%%%%%%%%%%%%%%%%%%%%%%%%%%%%%%%%%%%%%%%%%%%%%%%%%%%%%%%%%%%%%%%%%%%%%%%%%%%%%%%%%%%%%%%%%%%%%%%%%%%%%
%%%%%%%%%%%%%%%%%%%%%%%%%%%%%%%%%%%%%%%%%%%%%%%%%%%%%%%%%%%%%%%%%%%%%%%%%%%%%%%%%%%%%%%%%%%%%%%%%%%%%%%%%%%%%%%%
%%%%%%%%%%%%%%%%%%%%%%%%%%%%%%%%%%%%%%%%%%%%%%%%%%%%%%%%%%%%%%%%%%%%%%%%%%%%%%%%%%%%%%%%%%%%%%%%%%%%%%%%%%%%%%%%

\appendix

\section{Proof of theorem~\ref{theorem:Inequalities}}\label{Appendix:Inequalities}

\begin{proof}
It follows from the Schwarz--Christoffel construction that $\abs{r(z;-1,2)} \le1$. Next, by Skorokhodov's optimization argument, it is clearly true that:
\[
\abs{r(z;\gamma_{\rm opt},m)} \le \abs{r(z;-1,m)}.
\]
To show the first inequality in the chain, it will suffice to consider $n=m+1$. Observe that $r(1;\gamma,m) = 1$ for all $\gamma$ and $m\ge 1$. Thus, for all $z\ne1$, define $w=a+\ii b$ where $-\pi< b\le \pi$ so that:
\[
1-z = e^w = e^{a+\ii b}.
\]
Then the inequality is:
\[
\abs{\frac{1-e^{w/n}}{1+e^{w/n}}} \le \abs{\frac{1-e^{w/m}}{1+|e^{w/m}|}}.
\]
Equivalently:
\[
\frac{2\cos(b/n)}{\cosh(a/n)+\cos(b/n)} = 1 - \abs{\frac{1-e^{w/n}}{1+e^{w/n}}}^2 \ge 1 - \abs{\frac{1-e^{w/m}}{1+|e^{w/m}|}}^2 = \frac{\cos^2(b/2m)}{\cosh^2(a/2m)}.
\]
Clearing fractions, we find:
\[
2\cos(b/n)\left[1+\sinh^2(a/2m)\right] = 2\cos(b/n)\cosh^2(a/2m) \ge \left[\cosh(a/n)+\cos(b/n)\right]\cos^2(b/2m).
\]
Subtracting $\cosh(a/n)\cos^2(b/2m)$ from both sides:
\begin{align*}
\cos(b/n)\left[2-\cos^2(b/2m)\right] + 2\cos(b/n)\sinh^2(a/2m) & \ge \cos^2(b/2m)\cosh(a/n),\\
\Leftrightarrow\quad\cos(b/n)\left[2-\cos^2(b/2m)\right] + 2\cos(b/n)\sinh^2(a/2m) & \ge \cos^2(b/2m) + \cos^2(b/2m)\left[\cosh(a/n)-1\right],\\
\Leftrightarrow\quad\cos(b/n)\left[2-\cos^2(b/2m)\right] + 2\cos(b/n)\sinh^2(a/2m) & \ge \cos^2(b/2m) + 2\cos^2(b/2m)\sinh^2(a/2n).
\end{align*}
This inequality is true if the following two inequalities are separately true:
\begin{equation}\label{eq:InEqPart1}
\cos(b/n)\left[2-\cos^2(b/2m)\right] \ge \cos^2(b/2m),
\end{equation}
and:
\begin{equation}\label{eq:InEqPart2}
\cos(b/n)\sinh^2(a/2m) \ge \cos^2(b/2m)\sinh^2(a/2n).
\end{equation}
For Eq.~\eqref{eq:InEqPart1}:
\begin{align*}
\cos(b/n)\left[2-\cos^2(b/2m)\right] - \cos^2(b/2m) & \ge 0,\\
1+\cos(b/n)-\cos^2(b/2m)-\cos(b/n)\cos^2(b/2m) & \ge 1 - \cos(b/n),\\
\left[1+\cos(b/n)\right]\left[1-\cos^2(b/2m)\right] & \ge 1 - \cos(b/n),\\
\frac{1-\cos(b/n)}{1+\cos(b/n)} & \le 1-\cos^2(b/2m),\\
\tan^2(b/2n) & \le \sin^2(b/2m).
\end{align*}
Without loss of generality, we work with $0\le b\le\pi$ and $n = m+1$ to show that:
\[
\tan(b/(2m+2)) \le \sin(b/2m).
\]
With $m=1$, we find $\tan(b/4) \le \sin(b/2)$. This is equivalent to $2\cos(b/4)^2 \ge 1$, which is true.

With $m\ge 2$, letting $x = b/2m$ and $y = rx$ where $r = m/(m+1)$, since for $x>0$ and $0<y<\sqrt{2}$:
\[
\sin x \ge x-\frac{x^3}{6},\quad{\rm and}\quad \tan y \le \frac{y}{1-y^2/2},
\]
we have:
\begin{align*}
\left(x-\frac{x^3}{6}\right)\left(1-\frac{r^2x^2}{2}\right) & \ge rx,\\
\left(1-\frac{x^2}{6}\right)\left(1-\frac{r^2x^2}{2}\right) \ge 1-\left(\frac{1}{6}+\frac{r^2}{2}\right)x^2 & \ge r,
\end{align*}
or:
\[
b^2 \le \frac{4m^2}{(m+1)\left(\frac{1}{6}+\frac{m^2}{2(m+1)^2}\right)} = \frac{24(m+1)}{4+\frac{2}{m}+\frac{1}{m^2}}.
\]
Now, the numerator is increasing in $m$ while the denominator is decreasing. The inequality holds for all $b^2 \le 96/7$, which is greater than $\pi^2$.

For Eq.~\eqref{eq:InEqPart2}, since $\sinh^2(a/2m) \ge (n/m)^2\sinh^2(a/2n)$, we find:
\begin{equation}\label{eq:InEqPart2Trig}
n^2\cos(b/n) \ge m^2\cos^2(b/2m)
\end{equation}
Again, we work with $0\le b\le\pi$ and $n = m+1$ so that $\cos(b/n) \ge \cos(\pi/n)$ and $\cos^2(b/2m) \le 1$, or:
\[
\cos(\pi/n) \ge 1-\frac{\pi^2}{2n^2} \ge \frac{(n-1)^2}{n^2}.
\]
This is satisfied by $n \ge \frac{\pi^2+2}{4}$, a number slightly less than $3$. Observe that for $m = 1$ and $n=2$, Eq.~\eqref{eq:InEqPart2Trig} reads $4\cos(b/2) \ge \cos^2(b/2)$, or $4\ge\cos(b/2)$, which is clearly true.
\end{proof}

\section{On the ultraspherical spectral method}\label{appendix:ultrasphericalspectralmethod}

The ultraspherical spectral method of Olver and Townsend~\cite{Olver-Townsend-55-462-13,Olver-Slevinsky-Townsend-29-573-20} relies on representing differentiation, multiplication by \(x\), and conversion between certain orthogonal polynomial bases as sparse operators acting on coefficient sequences. In this appendix, we collect the basic relations used throughout the paper. These relations are standard consequences of the recurrence and differentiation formulas for Chebyshev and ultraspherical polynomials, but we state them in matrix form to fix notation and indexing conventions.

Let
\[
\bs{T}(x)
=
\begin{pmatrix}
T_0(x)&T_1(x)&T_2(x)&\cdots
\end{pmatrix}
\]
denote the first kind Chebyshev polynomial quasimatrix, and let
\[
\bs{C}^{(\lambda)}(x)
=
\begin{pmatrix}
C_0^{(\lambda)}(x)&C_1^{(\lambda)}(x)&C_2^{(\lambda)}(x)&\cdots
\end{pmatrix}
\]
denote the ultraspherical polynomial quasimatrix of order $\lambda>-\frac{1}{2}$, $\lambda\ne0$. Second kind Chebyshev polynomials are synonymous with ultraspherical polynomials of order $1$,
\(\bs{U}(x) \equiv \bs{C}^{(1)}(x)\).

\begin{proposition}[\S 18.9.1 and 18.9.2 in~\cite{NIST:DLMF}]\label{prop:XT}
Multiplication by $x$ of first kind Chebyshev polynomials:
\begin{equation}
x\bs{T}(x) = \bs{T}(x)\underbrace{\begin{pmatrix} 0 & \tfrac{1}{2}\\ 1 & 0 & \tfrac{1}{2}\\ & \tfrac{1}{2} & 0 & \ddots\\ & & \ddots & \ddots\end{pmatrix}}_{X_T}.
\end{equation}
\end{proposition}

\begin{proposition}[\S 18.9.1 and 18.9.2 in~\cite{NIST:DLMF}]\label{prop:Xlambda}
Multiplication by $x$ of ultraspherical polynomials:
\begin{equation}
x\bs{C}^{(\lambda)}(x) = \bs{C}^{(\lambda)}(x)\underbrace{\begin{pmatrix}
0 & \dfrac{2\lambda}{2(1+\lambda)} & & & \\
\dfrac{1}{2\lambda} & 0 & \dfrac{1+2\lambda}{2(2+\lambda)} & & \\
& \dfrac{2}{2(1+\lambda)} & 0 & \dfrac{2+2\lambda}{2(3+\lambda)} & \\
& & \dfrac{3}{2(2+\lambda)} & 0 & \ddots\\
& & & \ddots & \ddots
\end{pmatrix}}_{X_{(\lambda)}}.
\end{equation}
\end{proposition}

\begin{proposition}[\S 18.9.9 in~\cite{NIST:DLMF}]\label{prop:RTU}
Conversion of first kind Chebyshev polynomials:
\begin{equation}
\bs{T}(x) = \bs{U}(x)\underbrace{\begin{pmatrix} 1 & 0 & -\tfrac{1}{2}\\ & \tfrac{1}{2} & 0 & \ddots\\ & & \tfrac{1}{2} & \ddots\\ & & & \ddots\end{pmatrix}}_{R_T^U}.
\end{equation}
\end{proposition}

\begin{proposition}[\S 18.9.7 in~\cite{NIST:DLMF}]\label{prop:Rlambda}
Conversion of ultraspherical polynomials:
\begin{equation}
\bs{C}^{(\lambda)}(x) = \bs{C}^{(\lambda+1)}(x)\underbrace{\begin{pmatrix}
1 & 0 & -\dfrac{\lambda}{\lambda+2} &  &  & \\
  & \dfrac{\lambda}{\lambda+1} & 0 & -\dfrac{\lambda}{\lambda+3} &  & \\
  &  & \dfrac{\lambda}{\lambda+2} & 0 & -\dfrac{\lambda}{\lambda+4} & \\
  &  &  & \dfrac{\lambda}{\lambda+3} & 0 & \ddots\\
  &  &  &  & \ddots & \ddots
\end{pmatrix}}_{R_{(\lambda)}^{(\lambda+1)}}.
\end{equation}
\end{proposition}

\begin{proposition}[\S 18.9.21 in~\cite{NIST:DLMF}]\label{prop:DTU}
Differentiation of first kind Chebyshev polynomials:
\begin{equation}
\DD \bs{T}(x) = \bs{U}(x)\underbrace{\begin{pmatrix} 0 & 1\\ & & 2\\ & & & 3\\ & & & & \ddots\end{pmatrix}}_{D_T^U}.
\end{equation}
\end{proposition}

\begin{proposition}[\S 18.9.19 in~\cite{NIST:DLMF}]\label{prop:Dlambda}
Differentiation of ultraspherical polynomials:
\begin{equation}
\DD \bs{C}^{(\lambda)}(x) = \bs{C}^{(\lambda+1)}(x)\underbrace{\begin{pmatrix}
0&2\lambda& & & \\
 & &2\lambda& & \\
 & & &2\lambda& \\
 & & & &\ddots
\end{pmatrix}}_{D_{(\lambda)}^{(\lambda+1)}}.
\end{equation}
\end{proposition}

\end{document}